\documentclass[11pt]{article}
\usepackage[margin=1.1in,a4paper]{geometry}

\usepackage{mathtools}
\usepackage{amsmath, amssymb, natbib, graphicx, url, amsthm, subcaption}
\usepackage[ruled, lined]{algorithm2e}% linesnumbered
\usepackage[mathcal]{eucal}
\usepackage{tikz}
\usepackage{verbatim}

\usepackage{amsfonts}       % blackboard math symbols
\usepackage{enumitem}
\usepackage{microtype}      % microtypography
\usepackage{xcolor}
\usepackage[colorlinks=true,citecolor=blue]{hyperref}    % hyperlinks
\usepackage{diagbox}   
\usepackage{tcolorbox}
\usepackage{nicefrac}
\newcommand{\TT}{\mathbb{T}}

\newcommand{\RR}{\mathbb{R}}
\newcommand{\NN}{\mathbb{N}}

\renewcommand{\SS}{\mathbb{S}}

\newcommand{\Cc}{\mathcal{C}}

\newcommand{\Ff}{\mathcal{F}}

\newcommand{\Mm}{\mathcal{M}}

\newcommand{\Pp}{\mathcal{P}}

\renewcommand{\d}{\mathrm{d}}
\newcommand{\dd}{\, \mathrm{d}}

\newcommand{\Lip}{\mathrm{Lip}}

\DeclareMathOperator{\sign}{sign}
\DeclareMathOperator{\dom}{dom}

\DeclareMathOperator{\dist}{dist}

\newtheorem{theorem}{Theorem}[section]
\newtheorem{proposition}[theorem]{Proposition}
\newtheorem{lemma}[theorem]{Lemma}

\newtheorem{remark}[theorem]{Remark}

\DeclareMathOperator{\arcsinh}{arcsinh}
\DeclareMathOperator{\sfth}{sfth}
\graphicspath{ {./images/} } 

\title{Convergence Rates of Gradient Methods \\  for Convex Optimization in the Space of Measures\footnote{Accepted for publication at the Open Journal of Mathematical Optimization.}
}

\author{
L\'ena\"ic Chizat\thanks{CNRS, Universit\'e Paris-Saclay, Laboratoire de math\'ematiques d'Orsay, 91405, Orsay, France. \texttt{lenaic.chizat@universite-paris-saclay.fr}} 
}

\begin{document}
\maketitle

\begin{abstract}
We study the convergence rate of Bregman gradient methods for convex optimization in the space of measures on a $d$-dimensional manifold. Under basic regularity assumptions, we show that the suboptimality gap at iteration $k$ is in $O(\log(k)k^{-1})$ for multiplicative updates, while it is in $O(k^{-q/(d+q)})$ for additive updates for some $q\in \{1,2,4\}$ determined by the structure of the objective function. Our flexible proof strategy, based on approximation arguments, allows us to painlessly cover all Bregman Proximal Gradient Methods (PGM) and their acceleration (APGM) under various geometries such as the hyperbolic entropy and $L^p$ divergences. We also prove the tightness of our analysis with matching lower bounds and confirm the theoretical results with numerical experiments on low dimensional problems. Note that all these optimization methods must additionally pay the computational cost of discretization, which can be exponential in $d$.
\end{abstract}

\section{Introduction}
Convex optimization in the space of measures is a theoretical framework that leads to fruitful point of views on a large variety of problems, ranging from sparse deconvolution~\citep{bredies2013inverse} and two-layer neural networks~\citep{bengio2006convex} to global optimization~\citep{lasserre2001global} and many more~\citep{boyd2017alternating}. Various algorithms have been proposed to solve such problems including moments methods~\citep{lasserre2001global}, conditional gradient~\citep{bredies2013inverse,denoyelle2019sliding}, (non-convex) particle gradient flows~\citep{chizat2021sparse} and noisy versions~\citep{mei2018mean, nitanda2020particle}. 

In this paper, we consider perhaps the simplest methods: gradient descent and its extensions that handle non-smooth regularizers and non-Euclidean geometries, the Bregman Proximal Gradient Method (PGM) (an extension of mirror descent~\citep{nemirovskij1983problem} that handles composite objectives) and its acceleration (APGM)~\citep{tseng2010approximation}. Our aim is  {to establish well-posedness and convergence rates for these methods when minimizing composite functions over the space of measures $\Mm(\Theta)$ over a $d$-dimensional manifold $\Theta$, of the form
\begin{align}\label{eq:main-problem-intro}
F(\mu) \coloneqq R\left(\int \Phi\d\mu\right) + H(\mu)
\end{align}
where $\Phi$ is continuous and Hilbert space-valued, $R$ convex and smooth and $H$ is convex and ``simple'' (see precise assumptions in Section~\eqref{sec:objective}). For such problems,  minimizers are typically at an infinite (Bregman) distance from the initialization, and thus all the standard convergence bounds are inapplicable.} 

Our contributions are the following:
\begin{itemize}
\item We recall and adapt (A)PGM in Section~\ref{sec:optim-measure}, taking care of the subtleties that appear in our context (definition of the iterates and lack of strong convexity of the divergence);
\item We prove in Section~\ref{sec:upper-bounds} upper-bounds on the convergence rate for (A)PGM under various structural assumptions, summarized in Table~\ref{table:rates}. These rates depend on the choice of the Bregman divergence and on the precise structure of the objective function; 
\item Tight lower bounds of two kinds are proved in Section~\ref{sec:lower-bounds}: \emph{proof technique}-dependent lower bounds, and \emph{algorithm}-dependent lower bounds (the latter are stronger but do not cover all cases);
\item Numerical experiments on synthetic toy problems in Section~\ref{sec:numerics} often show an excellent agreement between the theoretical rates and the ones observed in practice. Even for cases with an apparent mismatch, a closer look at the structure of the problem shows that the theory still shades light on the observed rates.
\end{itemize}

Our motivation for studying this problem is threefold. First, our results make a case for APGM with the hyperbolic geometry instead of FISTA to solve convex problems in the space of measures, as they show that the former enjoys a faster convergence rate. Second, we believe that a precise understanding of (A)PGM in this context is useful to develop and analyze more complex methods, such as the particle-based (a.k.a.~moving grid) approaches mentioned above\footnote{In fact, the idea of writing this paper came from a technical step in a proof of~\citet{chizat2021sparse}, which studies particle-based methods.}. Third, this setting offers a rich test case to deepen our understanding of Bregman gradient methods in Banach spaces, {and the behavior of optimization algorithms when all minimizers are at an infinite distance from the initialization, beyond the well-explored Hilbert space setting.}

\paragraph{Related work}
The comparison between additive updates ($L^2$ geometry) and multiplicative updates (entropy geometry) is well-known in finite dimensional spaces~\citep{kivinen1997exponentiated}. For instance, for convex optimization in the $n$-dimensional simplex, the two methods typically converge at the same rate but the ``constant'' factor is polynomial in $n$ for additive updates while it is logarithmic in $n$ for multiplicative updates, see~\cite[Section~4]{bubeck2015convex}. We obtain in this paper an infinite dimensional ($n= \infty$) version of this separation; but where the distinction is directly in the rates rather than in the constants.

Analysis of convex optimization in infinite dimensional (Banach) spaces is a classical subject~\citep{bauschke2001essential,bauschke2003bregman}. Here, we study a concrete class of problems defined on the space of measures which exhibit specific features. This problem-specific approach for infinite dimensional problems has proved fruitful for the analysis of gradient methods for least-squares (e.g.~\cite{yao2007early,dieuleveut2017stochastic} and references therein), for partly smooth problems~\citep{liang2014local} and for the Iterative Soft Thresholding Algorithm (ISTA) in Hilbert spaces~\citep{bredies2008linear,garrigos2020thresholding}. 

The latter is close to our subject since ISTA is in fact an instance of PGM with the $L^2$-divergence -- and FISTA~\citep{beck2009fast} is analogous to APGM with the $L^2$-divergence. These prior works perform the analysis in a Hilbert space, while we work in the space of measures or in $L^1$, which are non-reflexive Banach space. This is also the context of~\citet{chambolle2021fista} who, for a modified version of FISTA, obtained in particular the convergence rate of Table~\ref{table:rates} when $p=2$ and $q=1$, and also discuss discretization. Our analysis allows to compare various algorithms and shows that FISTA is always slower than APGM with the hyperbolic entropy geometry~\citep{ghai2020exponentiated} when the solution is truly sparse, see the rates in Table~\ref{table:rates}. This is clearly observed in numerical experiments and suggests that the latter forms a stronger baseline for our class of problems.

To prove our upper bounds, we use the abstract proof strategy proposed by~\citet{jacobs2019solving}, recalled in Section~\ref{sec:general}. In that paper, the authors study different classes of problems (total variation denoising of image and earth mover's distance) under Hilbertian geometry.

\paragraph{Notation} The domain of a function $F:V\to \RR\cup \{+\infty\}$ is $\dom F = \{ x \in V\;;\; F(x)<+\infty\}$.
Throughout, $\Theta$ is a compact $d$-dimensional manifold, $\Mm(\Theta)$ (resp. $\Mm_+(\Theta)$) is the set of finite signed (resp.~nonnegative) Borel measures on $\Theta$ and  $\Pp(\Theta)$ is the set of Borel probability measures. For $\mu\in \Mm(\Theta)$, $\Vert \mu\Vert$ is its total variation norm. For a Hilbert space $\Ff$,  $\Cc^p(\Theta;\Ff)$ is the set of $p$-times continuously differentiable functions from $\Theta$ to $\Ff$. $\Lip(f)$ is the Lipschitz constant of a function $f$. For $\tau\in\Pp(\Theta)$ and $p\geq 1$, $L^p(\tau)$ is the space of (equivalence classes of) measurable functions $f:\Theta\to \RR$ such that $\int_\Theta \vert f(\theta)\vert^p\d\tau(\theta)<+\infty$ or, for $p=+\infty$, such that $\vert f\vert$ is $\tau$-almost everywhere bounded by some $K>0$. The asymptotic notation $a(k) \lesssim b(k)$ means that there exists $c>0$ independent of $k$ such that $a(k)\leq c \cdot b(k)$, and $a(k)\asymp b(k)$ means [$a\lesssim b$ and $b \lesssim a$].

\section{Strategy to derive upper bounds on convergence rates}\label{sec:general}
This section introduces the strategy, adapted from~\citep{jacobs2019solving}, that we adopt to derive upper bounds on the convergence rates. 

Let $F$ be a lower bounded convex function defined on a real vector space. Suppose that an iterative method designed to minimize $F$ initialized at $x_0\in \dom F$ generates a sequence $x_1,x_2,\dots \in \dom F$ that satisfies
\begin{align}\label{eq:general-guarantee}
F(x_k) - F(x) \leq \alpha_k\cdot D(x, x_0), \quad \forall x\in \dom F,\quad \forall k\geq 1,
\end{align}
where $(\alpha_k)_{k\in \NN^*}$ is a positive sequence converging to $0$ and $D$ is a \emph{divergence}, i.e.\ $D(x,x_0)\in [0,+\infty]$ and $D(x_0,x_0)=0$. Most first order methods enjoy guarantees of this form. For instance, PGM and APGM enjoy such guarantees with respectively $\alpha_k \lesssim k^{-1}$ and $\alpha_k \lesssim k^{-2}$ under suitable assumptions, see Section~\ref{sec:PMD}. 

While Eq.~\eqref{eq:general-guarantee} is sometimes the endpoint of the analysis in the optimization literature, this is our starting point: we are interested in cases where {for any minimizer $x^*$ of $F$, the quantity $D(x^*, x_0)$ is infinite, which makes the bound of Eq.~\eqref{eq:general-guarantee} inapplicable. Even if there exists a quasi-minimizer $x$ with a small suboptimality gap and satisfying $D(x,x_0)<+\infty$,} choosing a fixed $x$ independent of $k$ in Eq.~\eqref{eq:general-guarantee} leads to a poor upper bound which often does not match the observed practical behavior.
Instead, we should exploit the flexibility offered by the guarantee of Eq.~\eqref{eq:general-guarantee} and choose a different reference point at each time step\footnote{{In Section~\ref{sec:upper-bounds}, these reference points will be constructed as mollifications of the optimal measure $\mu^*$.}}. This means that we reformulate the guarantee in the equivalent form:
\begin{align}\label{eq:general-rate}
F(x_k) - \inf F \leq \psi(\alpha_k) \quad\text{where}\quad \psi(\alpha) \coloneqq \inf_{x} \Big\{F(x) - \inf F + \alpha  D(x, x_0)\Big\}. 
\end{align}
Studying $\psi$ is particularly fruitful to understand optimization algorithms satisfying Eq.~\eqref{eq:general-guarantee}. In particular, its behavior at $0$ determines the asymptotic convergence rate. This function can be interpreted as the value at $x_0$ of the (Bregman) Moreau envelope~\citep{kan2012moreau} of $(F- \inf F)$ with regularization parameter $\alpha$, and it intervenes in many area of applied mathematics. For instance, when $D(x,x_0)$ is a squared Hilbertian norm, $\psi$ has a variety of behaviors which characterize the performance of kernel ridge regression in machine learning (see e.g.~\cite[Chap.~7.5]{bach2021learning}).
Before we head in a more concrete setting, let us gather a few relevant properties of the function $\psi$ that hold in full generality.

 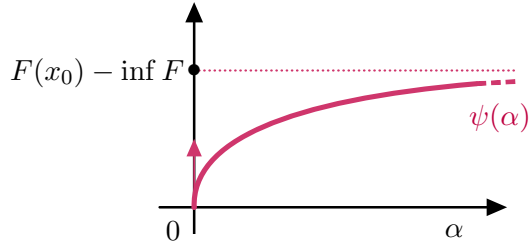
\begin{figure}
 \centering
\begin{tikzpicture}
\node[inner sep=0pt] (russell) at (0,0)
    {\includegraphics[scale=0.8,trim={4cm 23cm 5cm 1cm},clip]{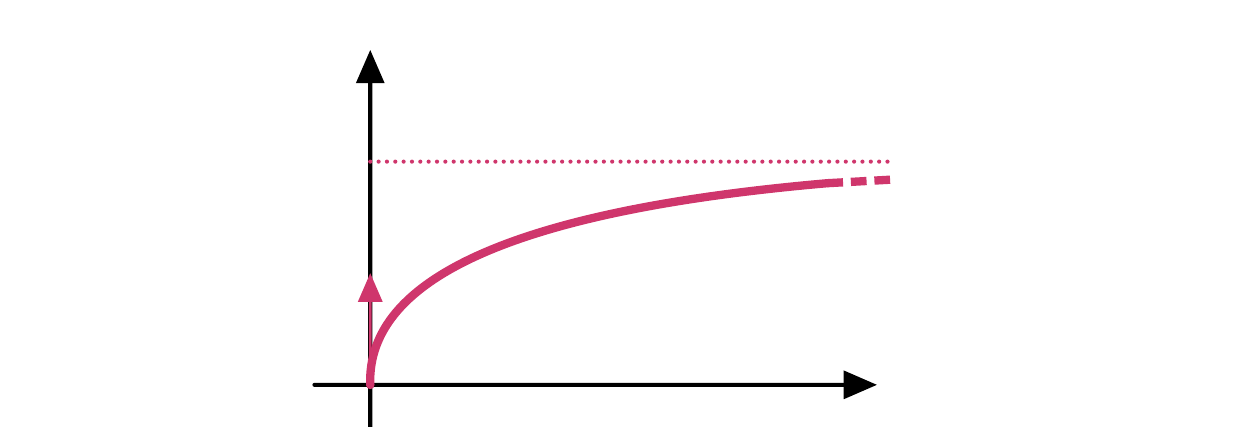}};
\node (0) at (-2.3,-1.7) {$0$};
\node (0) at (1.4,-1.7) {$\alpha$};
\node (0) at (-3.3,0.4) {$F(x_0)-\inf F$};
\node (0) at (-2.025,0.42) {$\bullet$};
\node[color=purple] (0) at (2,-0.2) {$\psi(\alpha)$};
\end{tikzpicture}
\caption{Shape of $\psi$ defined in Eq.~\eqref{eq:general-rate} in our situation of interest where $D(x_k,x_0)$ explodes for any minimizing sequence $(x_k)_{k\in\NN}$ (Prop.~\ref{prop:general-psi}). When an optimization method satisfies Eq.~\eqref{eq:general-guarantee} for some sequence $(\alpha_k)_{k\in \NN}$ then $\psi(\alpha_k)$ bounds its convergence rate in objective values.}
\label{fig:shape}
\end{figure}

\begin{proposition}\label{prop:general-psi}
Assume that $F(x_0)<+\infty$ and $D(\cdot,x_0)\geq 0$ with equality at $x_0$. Then the function $\psi$ is concave on $[0,+\infty[$ and satisfies $0 = \psi(0)\leq \psi(\alpha)\leq F(x_0)-\inf F$.  Moreover,

 (i) $\psi$ is right-continuous at $0$ if and only if there exists a minimizing sequence $(x_k)_{k\in \NN}$ such that $F(x_k)\to \inf F(x)$ and $D(x_k,x_0)<+\infty$, $\forall k\in \NN$ ;

(ii) $\psi'(0)\coloneqq \lim_{\alpha\to 0^+} \psi(\alpha)/\alpha$ is finite if and only if there exists a minimizing sequence $(x_k)_{k\in \NN}$ such that $F(x_k)\to \inf F(x)$ and $D(x_k,x_0)$ is bounded.
\end{proposition}
\begin{proof}
The function $\psi$ is concave as the pointwise infimum of affine functions. The lower bound is immediate and the upper bound is obtained by taking $x_0$ as a candidate in the infimum. Let us prove (ii) (the proof of (i) follows a similar scheme and is simpler). By concavity, the limit defining $\psi'(0)$ always exists and belongs to ${]0,+\infty]}$. If a sequence $(x_k)$ exists as in the statement, then for any $\alpha>0$, {pick $x_k$ such that $F(x_k)-\inf F\leq \alpha^2$ and then $\psi(\alpha)\leq \alpha D(x_k,x_0) + \alpha^2$ so $\psi(\alpha)/\alpha \leq D(x_k,x_0) +\alpha$}. Since the upper bound is uniformly bounded as $\alpha\to 0$ it follows that $\psi'(0)$ is finite. Conversely, if $\psi'(0)$ is finite, take a decreasing sequence $(\alpha_k)$ that converges to $0$ and let $(x_k)$ be a sequence of quasi-minimizers for Eq.~\eqref{eq:general-rate} satisfying $F(x_k) - \inf F +\alpha_k D(x_k,x_0)\leq \psi(\alpha_k) +\alpha^2_k$. {By dividing by $\alpha_k$, we see that $(F(x_k)-\inf F)/\alpha_k + D(x_k,x_0)$ is bounded as $k\to \infty$ which implies that $F(x_k)\to \inf F$ and $D(x_k,x_0)$ is bounded.} 
\end{proof}

Figure~\ref{fig:shape} illustrates the general shape of the function $\psi$. { Observe that if $\psi'(0)<+\infty$ then the bound of Eq.~\eqref{eq:general-rate} is $F(x_k)-\inf F\leq\psi'(0)\alpha_k+o(\alpha_k)$ and thus the convergence rate given by Eq.~\eqref{eq:general-guarantee} is not modified (only the constant changes).
However, Proposition~\ref{prop:general-psi} shows that when any minimizing sequence $(x_k)$ satisfies $D(x_k,x_0)\to +\infty$, then $\psi'(0)=+\infty$ and thus the \emph{convergence rate} is modifed. This is the situation we are interested in in this paper, in the context of optimization in the space of measures}.

\section{Gradient methods for optimization in the space of measures}\label{sec:optim-measure}
In the rest of this paper, we apply the general method of Section~\ref{sec:general} to a class of optimization problems in the space of measures where it leads to a zoo of -- often tight -- convergence rates. 

\subsection{Objective function}\label{sec:objective}
Let $\Theta$ be a compact Riemannian manifold without boundary, with distance $\dist$ and with a reference probability measure $\tau \in \Pp(\Theta)$ that is proportional to the volume measure. We consider an objective function on the space of measures $\bar F:\Mm(\Theta)\to \RR\cup \{+\infty\}$ of the form
\begin{align*}
\bar F(\mu) \coloneqq \bar G(\mu) + \bar H(\mu)&&\text{where}&& \bar G(\mu) \coloneqq R\left(\int \Phi\d \mu\right).
\end{align*}
Typically, $\bar G$ is a data-fitting term and $\bar H$ a regularizer. We make the following assumptions, where $\iota_C$ is the convex indicator of a convex set $C$ and $\lambda \geq 0$ a regularization parameter:
\begin{itemize}
\item[\textbf{(A1)}] $\Phi \in \Cc^0(\Theta;\Ff)$ where $\Ff$ is a Hilbert space, $R:\Ff\to \RR$ is convex and differentiable with a Lipschitz gradient $\nabla R$, and $\bar H$ is a sum of functions from the following list: $\iota_{\Pp(\Theta)}$, $\iota_{\Mm_+(\Theta)}$, $\lambda \Vert \mu\Vert $ and $\iota_{\{\mu\;;\; \lambda\Vert \mu\Vert\leq 1 \}}$. 
\end{itemize}
One specific property of $\bar H$ that we use in our proof is that it should be non-decreasing under convolutions by a probability kernel, but we prefer to work with these specific instances rather than giving abstract conditions. We finally denote by $F:L^1(\tau)\to \RR\cup \{+\infty\}$ the function defined, for $f\in L^1(\tau)$, by
\begin{align*}
F(f )\coloneqq \bar F(f\tau),
\end{align*}
and similarly $H(f)\coloneqq \bar H(f\tau)$ and $G(f) \coloneqq \bar G(f\tau)$ so that $F = G + H$. These ``bar'' notations convey the idea that $\bar F,\bar G,\bar H$ are the lower-semicontinuous (l.s.c.) extensions of $F,G,H$ for the weak* topology induced by $\Cc^0(\Theta)$ on $\Mm(\Theta)$. 

Here are examples of problems that fall under this setting:
\begin{itemize}
\item (Sparse deconvolution) The goal is, given a signal $y^*\in L^2(\tau)$, to find a sparse measure $\mu$ such that the convolution of $\mu$ with a filter $\phi \in \Cc(\Theta)$ approximately recovers $y^*$. Here the domain is typically the $d$-dimensional torus $\Theta = \mathbb{T}^d$ endowed with the Lebesgue measure $\tau$ and the objective is~\citep{de2012exact,candes2014towards}
\begin{align*}
\bar F(\mu) \coloneqq \int_{\Theta} \Big\vert \int_{\Theta} \phi(\theta_1-\theta_2)\d\mu(\theta_2) - y^*(\theta_1)\Big\vert^2 \d\tau(\theta_2)  + \lambda \Vert \mu\Vert.
\end{align*}
Adding the nonnegativity constraint $\iota_{\Mm_+(\Theta)}$ is also relevant in certain applications.
\item (Two-layer relu neural networks). The goal is, given $n$ observations $(x_i,y_i)\in \RR^d\times \RR$, to find a regressor written as a linear combination of simple ``ridge'' functions. Consider a loss $\ell:\RR^2\to \RR$ convex and smooth in its second argument, let $\phi(s)=(s)_+$ and let
\begin{align*}
\bar F(\mu) \coloneqq \frac1n \sum_{i=1}^n \ell\Big(y_i, \int_{\Theta} \phi([x_i;1]^\top \theta)\d\mu(\theta) \Big) + \bar H(\mu).
\end{align*}
Key differences with the previous setting are that $\Theta = \SS^{d}$ is the sphere, with $d$ potentially large, and that the object that is truly sought after is the regressor $x\mapsto \int \phi([x;1]^\top \theta)\d\mu(\theta)$ rather than the measure $\mu$. Typical choices for $\ell$ are the logistic loss $\ell(y,z)= \log(1+\exp(-yz))$ when $y_i\in \{-1,1\}$ or the square loss $\ell(y,z)=\frac12\vert y-z\vert^2$. The signed setting with regularization $\bar H(\mu)=\lambda \Vert \mu\Vert$ is the most common one~\citep{bengio2006convex,bach2017breaking} but the regularization $\iota_{\Pp(\Theta)}$ also appears in the context of max-margin problems~\citep{chizat2020implicit}. 
\end{itemize}

The following smoothness lemma will be useful to analyze optimization algorithms and is analogous to the usual ``Lipschitz gradient'' property in convex optimization. Since the dual of $(\Mm(\Theta),\Vert \cdot\Vert)$ is a bit exotic, we avoid using the notion of gradient at all.
\begin{lemma}[Smoothness]\label{lem:smoothness} Under Assumption (A1), if $\Phi\in \Cc^p(\Theta,\Ff)$ for $p\in \NN$, then the differential of $\bar G$ at $\mu \in \Mm(\Theta)$ can be represented by the function $\bar G'[\mu]\in \Cc^p(\Theta,\RR)$ defined by
\begin{align*}
\bar G'[\mu](\theta) = \Big\langle \nabla R\Big(\int \Phi \d\mu \Big), \Phi(\theta)\Big\rangle_\Ff
\end{align*}
in the sense that it holds $\bar G(\mu+\nu) - \bar G(\mu) = \int \bar G'[\mu](\theta) \d\nu (\theta) + o(\Vert \mu -\nu \Vert)$. Moreover, $\mu \mapsto \bar G'[\mu]$ is Lipschitz continuous as a function from $\Mm(\Theta)$ to $\Cc^p(\Theta,\RR)$. The following smoothness inequality holds with $\Lip(\bar G')\leq \Vert \Phi\Vert_\infty^2\cdot \Lip(\nabla R)$ and for all $\mu,\nu \in \Mm(\Theta)$,
\begin{align*}
0\leq \bar G(\nu) - \bar G(\mu) - \int \bar G'[\mu]\d[\nu-\mu] \leq \frac12 \Lip(\bar G')\Vert \nu-\mu\Vert^2.
\end{align*}
Those results hold true when replacing $(\Mm(\Theta),\bar G(\mu),\bar G'[\mu])$ by $(L^1(\tau),G(f),G'[f]\coloneqq \bar G'[f\tau])$.  
\end{lemma}

\begin{proof}
For the first part, the differentiability of $R$ implies that
\begin{align*}
\bar G(\mu+\nu) - \bar G(\mu) &= \Big\langle \nabla R\Big(\int \Phi \d\mu\Big), \int \Phi \d \nu \Big\rangle_\Ff + o\Big(\Big\Vert \int \Phi \d\nu\Big\Vert_{\Ff}\Big) = \int \bar G'[\mu] \d\nu+ o(\Vert \Phi\Vert_\infty \Vert \nu \Vert).
\end{align*}
For the regularity of $\mu \mapsto \bar G'[\mu]$, we have for $\mu,\nu\in \Mm(\Theta)$,
\begin{align*}
\Vert \bar G'[\mu] - \bar G'[\nu]\Vert_{\Cc^p} \leq \Vert \Phi\Vert_{\Cc^p(\Theta,\Ff)} \cdot  \Lip(\nabla R)\cdot \Vert \Phi\Vert_\infty\cdot \Vert \mu -\nu\Vert.
\end{align*}
The smoothness inequality can be shown by bounding a $1$-dimensional integral as in the Euclidean case~\cite[Thm.~2.1.5]{nesterov2003introductory}. Finally, $L^1(\tau) \ni f \mapsto f\tau \in \Mm(\Theta)$ is an isometry, so those results hold mutandis mutatis in $L^1(\tau)$.
\end{proof}

\subsection{Bregman divergences} 
Let us consider $\eta:\RR\to [0,\infty]$ a differentiable  function {that we will refer to as the \emph{distance-generating} function}. For $f\in L^1(\tau)$ we write $\eta(f)\coloneqq \eta \circ f$ and we define
\begin{align*}
\bar \eta(f) \coloneqq \int_\Theta \eta(f(\theta))\dd\tau(\theta) = \int \eta(f)\d\tau.
\end{align*}
Let $D_\eta$ (resp. $D_{\bar \eta}$) be the Bregman divergence associated to $\eta$ (resp. $\bar \eta$), given for $f,g\in L^1(\tau)$ by
\begin{align*}
D_\eta(a,b) \coloneqq \eta(a) - \eta(b) - \eta'(b)\big(a-b\big)
&&\text{and}&&
D_{\bar \eta}(f,g) \coloneqq \int_\Theta D_\eta(f(\theta),g(\theta))\d\tau(\theta).
\end{align*}
We consider the following assumptions on the distance-generating function $\eta$:
\begin{itemize}
\item[\textbf{(A2)}] $\eta:\RR\to [0,\infty]$ is strictly convex, l.s.c., continuously differentiable in $\mathrm{int} (\dom \eta)$, such that $\eta'(\mathrm{int} (\dom \eta)) = \RR$ and {for any $c>0$ it holds $\eta(cx)\asymp \eta(x)$ as $x\to \infty$}. Moreover, either:
\begin{itemize}
\item[$\textbf{(A2)}_+$] $\dom \eta = [0,+\infty[$ and $\eta(1)=\eta'(1)=0$, or
\item[$\textbf{(A2)}_\pm$] $\dom \eta = \RR$, $\eta$ is even and $\eta(0)=\eta'(0)=0$.
\end{itemize}
\end{itemize}
Specifying the values of $\eta$ and $\eta'$ at a point in $\mathrm{int}(\dom \eta)$ is just for convenience and is not restrictive since $D_\eta$ is not affected by affine perturbations of $\eta$. {Also, the assumption $\eta(cx)\asymp \eta(x)$ is only needed to simplify the statement of the results}. Under assumption $\text{(A2)}_+$, we have that $\eta'(0)\coloneqq \lim_{t\to 0^+}\eta(t)/t = -\infty$ which automatically enforces an nonnegativity constraint in the methods in the next section.

Here are examples of distance-generating functions that fall under these assumptions:
\begin{itemize}
\item (Power functions $\eta_{p}$). Defined on $\RR$ for $p> 1$ by $\eta_p(s) \coloneqq  \frac{\vert s\vert^p}{p(p-1)} $, satisfy $\text{(A2)}_\pm$;
\item (Shannon entropy $\eta_{\mathrm{ent}}$). Defined on $\RR_+$ by $\eta_{\mathrm{ent}}(s)\coloneqq s\log(s)-s+1$, satisfies $\text{(A2)}_+$;
\item (Hyperbolic entropy $\eta_{\mathrm{hyp}}$). Defined on $\RR$ by $\eta_{\mathrm{hyp}} \coloneqq s \arcsinh (s/\beta)-\sqrt{s^2+\beta^2}+\beta$ with $\beta>0$, satisfies $\text{(A2)}_\pm$ (introduced by~\citep{ghai2020exponentiated}).
\end{itemize}

\begin{figure}
\centering
\includegraphics[width=\linewidth]{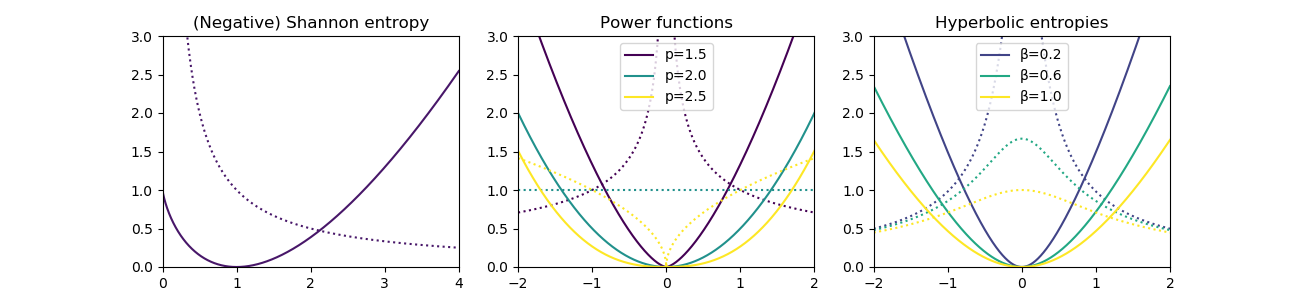}
\caption{Divergence-generating functions (plain) and their second-order derivatives (dashed).}\label{fig:divergences}
\end{figure}

When $\eta$ is smooth, it holds $D_\eta(a,b)=\eta''(b)\cdot \Vert a-b\Vert^2/2 + o(\Vert a-b\Vert^2)$ so locally, $D_\eta$ is equivalent to a squared Riemannian metric on the real axis given by $\eta''$. For the examples listed above, it holds $\eta''_p(s)=\vert s\vert^{p-2}$, $\eta''_\mathrm{ent}(s)=s^{-1}$ and $\eta''_\mathrm{hyp}(s)=(s^2+\beta^2)^{-1/2}$, see Figure~\ref{fig:divergences} for an illustration.
The hyperbolic entropy $\eta_{\mathrm{hyp}}$ can be interpreted as a ``signed'' version of $\eta_{\mathrm{ent}}$ (see Proposition~\ref{prop:reparametrized} for a precise version of this remark).

The next lemma states the strong convexity of these divergences with respect to the $L^1(\tau)$ norm, which is needed in the next section. It is a generalization of Pinsker inequality, recovered when $K=1$ and with $\eta_{\mathrm{ent}}$. Notice that when $p<2$, the bound worsens as the norm increases.
{
\begin{lemma}[Strong convexity of Bregman divergences]\label{lem:strong_convexity}
Assume that $f,g$ have $L^1(\tau)$-norm bounded by $K$. Then for $p \in {]1,2]}$,
\begin{align*}
D_{\bar \eta_p}(f,g) \geq \frac{K^{p-2}}{2} \Vert f-g\Vert^2_{L^1(\tau)}. 
\end{align*}
The inequality also holds for $\eta_{\mathrm{ent}}$ with $p=1$ (assuming $f,g\geq 0$). Finally for $\eta_{\mathrm{hyp}}$, it holds  $$D_{\bar \eta_{\mathrm{hyp}}}(f,g) \geq \frac{(K+\beta)^{-1}}{2} \Vert f-g\Vert^2_{L^1(\tau)}.$$ 
\end{lemma}

\begin{proof}
For $\beta>0$ and $p\in [1,2]$, consider the function $\eta_{p,\beta}:\RR\to \RR_+$ satisfying $\eta_{p,\beta}(0)=\eta_{p,\beta}'(0)=0$ and $\eta_{p,\beta}''(s)=\vert s^2+\beta^2\vert^{\frac{p-2}{2}}$. This function is smooth and, for $p>1$, converges monotonously from below to $\eta_p$ as $\beta\to 0$ (remember that $\eta_p$ satisfies $\eta_p(0)=\eta_p'(0)=0$ and $\eta''_p(s)=|s|^{p-2}\geq \eta''_{p,\beta}(s)$). Our first step is to prove a Pinsker-like inequality for the Bregman divergence $D_{\bar \eta_{p,\beta}}$. For $f,g\in L^1(\tau)$ such that $\Vert g\Vert_{L^1(\tau)}=1$, it holds by the Cauchy-Schwarz inequality
\begin{align*}
1 = \Big(\int \vert g\vert\d\tau\Big)^2 &= \Big(\int \vert g\vert\vert f^2+\beta^2\vert^{\frac{p-2}{4}}\vert f^2+\beta^2\vert^{\frac{2-p}{4}}\d\tau\Big)^2\\
&\leq \Big(\int \vert g\vert^2 \vert f^2+\beta^2\vert^{\frac{p-2}{2}}\d\tau \Big)\Big(\int\vert f^2+\beta^2\vert^{\frac{2-p}{2}}\d\tau\Big).
\end{align*}
But since $2-p\geq 0$, we have $\vert f^2+\beta^2\vert^{\frac{2-p}{2}} \leq (\vert f\vert+\beta)^{2-p}$. Thanks to Jensen's inequality for concave functions (we use $2-p\leq 1$), this leads to
\begin{align*}
\int\vert f^2+\beta^2\vert^{\frac{2-p}{2}}\d\tau \leq \int(\vert f\vert+\beta)^{2-p}\d\tau \leq (\Vert f\Vert_{L^1}+\beta)^{2-p}.
\end{align*}
Combining these inequalities and by homogeneity, it follows that if $\Vert f\Vert_{L^1(\tau)}\leq K$, then $\forall g \in L^1(\tau)$ it holds
\begin{align*}
\int \vert g\vert^2 \vert f^2+\beta^2\vert^{\frac{p-2}{2}} \d\tau\geq  (K+\beta)^{p-2} \Vert g\Vert^2_{L^1(\tau)}.
\end{align*}
This equation shows that $f\mapsto \bar \eta_{p,\beta}(f) \coloneqq \int \eta_{p,\beta}(f)\d\tau$ is $(K+\beta)^{p-2}$-strongly convex for the $L^1(\tau)$-norm over the $L^1(\tau)$-ball of radius $K$, see e.g.~\cite[Thm.~3]{yu2013strong}. This means that for all $f,g\in L^1(\tau)$ of norm smaller than $K$, it holds
$$
\lambda  \bar \eta_{p,\beta}(f) + (1-\lambda)  \bar \eta_{p,\beta}(g) \geq  \bar \eta_{p,\beta}(\lambda f+(1-\lambda)g) +\frac{(K+\beta)^{p-2}}2 \lambda(1-\lambda)\Vert f-g\Vert_{L^1(\tau)}^2,\quad
\forall \lambda \in [0,1].$$
 By the monotone convergence theorem, we have when $p>1$ that $\lim_{\beta \to 0}\bar \eta_{p,\beta}(f)\to \bar \eta_{p}(f)$ and so strong convexity also holds for $\bar \eta_{p}$, taking the pointwise limit of the strong convexity inequality as $\beta\to0$.

It follows~\citep[Thm.~1]{yu2013strong} that, over this ball of functions, the Bregman divergence $D_{\bar \eta_{p,\beta}}$ (for $p\in ]1,2]$ and $\beta>0$ or for $p=1$ and $\beta>0$) satisfies the Pinsker-like inequality 
\begin{align}\label{eq:general-pinsker}
D_{\bar \eta_{p,\beta}}(f,g) \geq \frac{(K+\beta)^{p-2}}{2} \Vert f-g\Vert_{L^1(\tau)}.
\end{align}

Specializing to $p=1$ and $\beta>0$ proves the Lemma for $\eta_{\mathrm{hyp}}$ and specializing to $p>1$ and $\beta=0$ proves the Lemma for $\eta_p$ with $p\in ]1,2]$. It only remains to prove the case $\eta_{\mathrm{ent}}$, i.e.~the classical Pinsker inequality. Note that here we cannot take the limit $\beta\to 0$ because $\eta_{1,\beta}$ does not converge to $\eta_{\mathrm{ent}}$ (in fact it diverges, except at $s=0$). One way to recover this case in an analogous way, is to define instead $ \eta_{1,\beta}:\RR_+\to \RR_+$ as the function satisfying $ \eta_{1,\beta}(1)=  \eta_{1,\beta}'(1)=0$ and $\eta_{1,\beta}''(s)=\vert s^2+\beta^2\vert^{-1/2}$. This function is smooth and converges monotonously to $\eta_{\mathrm{ent}}$ as $\beta \to 0$ and $\bar \eta_{1,\beta}$ is $(K+\beta)^{-1}$ strongly convex. Thus we recover Pinsker's inequality with an analogous argument in the limit $\beta\to 0$.
\end{proof}
}

\subsection{Gradient methods and their classical guarantees}\label{sec:PMD}
We now detail two classical algorithms that enjoy guarantees of the form Eq.~\eqref{eq:general-guarantee} for a large class of composite optimization problems. Algorithm~\ref{alg:PGM} (PGM) is closely related to mirror descent~\citep{nemirovskij1983problem} and is discussed in~\citet{bauschke2003bregman,auslender2006interior}. Algorithm~\ref{alg:APGM} (APGM) is taken from~\citet{tseng2010approximation} who presents it as a generalization of~\citet{auslender2006interior} itself an extension of Nesterov's second method~\citep{nesterov1988approach}.  For the sake of concreteness, we instantiate these algorithms in the context of optimization in the space of measures, where small adaptations have to be made.

\begin{algorithm}
\DontPrintSemicolon
%\SetAlgoLined
%\SetKwInOut{Input}{Input}\SetKwInOut{Output}{Output}
\SetKw{Initialization}{Initialization:}
\SetKw{Output}{Output:}
\Initialization{$f_0\in \dom H$, step-size $s>0$}\;
% initialization\;
 \For{k=0,1,\dots}{
 $ f_{k+1} = \arg\min_{f}  \big\{ G(f_k) + \int G'[f_k] (f-f_k)\d\tau + H(f) + \frac{1}{s} D_{\bar \eta}(f,f_k)\big\}$\;
 }
  \Output{$f_{k+1}$}
 \caption{(Bregman) Proximal Gradient Method (PGM)}\label{alg:PGM}
\end{algorithm}

\begin{algorithm}
\DontPrintSemicolon
%\SetAlgoLined
%\SetKwInOut{Input}{Input}\SetKwInOut{Output}{Output}
\SetKw{Initialization}{Initialization:}
\SetKw{Output}{Output:}
\Initialization{$f_0=h_0\in \dom H$, $\gamma_0=1$, step-size $s>0$}\;
% initialization\;
 \For{k=0,1,\dots}{
 $ g_k = (1-\gamma_k)f_k + \gamma_k h_k$\;
 $ h_{k+1} = \arg\min_{f}  \big\{ G(g_k) + \langle \nabla G(g_k), f-g_k\rangle + H(f) + \frac{\gamma_k}{s} D_{\bar \eta}(f,h_k)\big\}$\;
 $f_{k+1} = (1-\gamma_k)f_k + \gamma_k h_{k+1}$\;
 $\gamma_{k+1} =\frac12  \big(\sqrt{\gamma_k^4+4\gamma_k^2}-\gamma_k^2\big)$
 }
  \Output{$f_{k+1}$}
 \caption{Accelerated (Bregman) Proximal Gradient Method (APGM)}\label{alg:APGM}
\end{algorithm}

 In the next proposition, we verify that the updates are well-defined under suitable assumptions. Table~\ref{table:update} lists some update formulas which are directly implementable, after discretization.
\begin{proposition}[Well-defined updates]\label{prop:update}
Assume (A1) and (A2). If $\eta'(h_k) \in L^\infty(\tau)$ and $g_k\in L^1(\tau)$, then there exists a unique solution $h_{k+1} \in L^1(\tau)$ to the optimization problem
\begin{align*}
\min_{f\in L^1(\tau)} \int G'[g_k]f\d\tau + H(f) +\frac{1}{s} \Big( \int \eta(f)\d\tau - \int \eta'(h_k)f\d\tau \Big)
\end{align*}
which moreover satisfies $\eta'(h_{k+1}) \in L^{\infty}(\tau)$. It is characterized by the fact that there exists  $\phi \in \partial H(h_{k+1})\subset L^\infty(\tau)$ such that 
\begin{align}\label{eq:optim-condition-update}
G'[g_k] +\frac1s (\eta'(h_{k+1}) - \eta'(h_{k}) ) + \phi =0.
\end{align}

\end{proposition}
\begin{proof}
Let $J_k$ be the function to minimize. Thanks to our assumptions that $\eta'(\mathrm{int} (\dom \eta)) = \RR$ and since $H$ is lower-bounded, by~\cite[Cor.~2B]{rockafellar1971integrals}, the sublevels of $J_k$ are compact with respect to the weak topology (induced on $L^1(\tau)$ by $L^\infty(\tau)$). Moreover, $J_k$ is convex and l.s.c.~for the same topology; in particular because $G'[g_k],\eta'(h_k)\in L^\infty(\tau)$ and for the term $\int \eta(f)\d\tau$, this follows from~\cite[Cor.~2A]{rockafellar1971integrals}. Thus by the direct method of the calculus of variations, there exists a minimizer $f_{k+1}\in L^1(\tau)$. Since $\eta$ is strictly convex, so is $J_k$ and this minimizer is unique. The condition of Eq.~\eqref{eq:optim-condition-update} is always a sufficient optimality condition since, by the subdifferential inclusion rule, it implies that $0\in \partial J_k(f)$. It thus remain to show that it is also necessary, in which case the property $\eta'(h_{k+1})\in L^\infty(\tau)$ immediately follows. 

This is done on a case by case basis for the functions $\bar H$ admissible under Assumption~(A1). Consider for instance the nonnegativity constraint $\bar H = \iota_{\Mm_+(\Theta)}$ and $\eta$ that satisfies Assumption~$\text{(A2)}_\pm$. Then with the update $h_{k+1}$ given in Table~\ref{table:update} (take $\lambda=0$), it holds
\begin{align*}
\phi &\coloneqq \frac{1}{s}\eta'(h_{k})-G'[g_k] -\frac1s \eta'(h_{k+1}) = \min\{ 0 , \frac{1}{s}\eta'(h_{k})-G'[g_k]\}.
\end{align*}
Clearly $\phi \in L^\infty(\tau)$, $\phi \leq 0$ and $\int \phi h_{k+1}\d\tau =0$ and thus $\phi \in \partial H(h_{k+1})$, which shows that $h_{k+1}$ is a minimizer and satisfies Eq.~\eqref{eq:optim-condition-update}. The other cases for $\bar H$ and $\eta$ that are admissible under (A1) and (A2) (such as those listed in Table~\ref{table:update}) can be treated similarly and follow computations which are standard in the finite dimensional setting.
\end{proof}

\begin{table}
\centering
\begin{tabular}{|l|c|c|}
\hline
\diagbox[width=2.9cm]{ $\bar H$}{$\eta$} & Assumption $\text{(A2)}_\pm$ ($\dom \eta=\RR$) & Assumption $\text{(A2)}_+$ ($\dom \eta =\RR_+$) \\
\hline
(i) $\iota_{\Mm_+(\Theta)} + \lambda \Vert \cdot\Vert$ &$ [\eta']^{-1}\big( \eta'(h_k) - sG'[g_k] -s\lambda \big)_+$ &$ [\eta']^{-1}\big( \eta'(h_k) - sG'[g_k] -s\lambda \big) \vphantom{\Big(}$\\
(ii) $\iota_{\Pp(\Theta)}$ &$ [\eta']^{-1}\big( \eta'(h_k) - sG'[g_k] -\kappa \big)_+$ &$ [\eta']^{-1}\big( \eta'(h_k) - sG'[g_k] -\kappa \big)\vphantom{\Big(}$\\
(iii) $\lambda \Vert \cdot\Vert$ &$ [\eta']^{-1}\Big(\sfth_{\lambda s}\big( \eta'(h_k) - sG'[g_k] \big)\Big)$ & $ [\eta']^{-1}\big( \eta'(h_k) - sG'[g_k] -s\lambda \big) \vphantom{\Big(}$ \\
(iv) $\iota_{\{\mu\;; \Vert \mu\Vert\leq K \}}$ &$ [\eta']^{-1}\Big(\sfth_{\kappa}\big( \eta'(h_k) - sG'[g_k]  \big)\Big)$ & $ [\eta']^{-1}\big( \eta'(h_k) - sG'[g_k] -\kappa \big)\vphantom{\Big(}$ \\
\hline
\end{tabular}
\caption{Update step $h_{k+1}$ for APGM, and PGM (when $f_k=g_k=h_k$) where $\sfth_\kappa(a) = \sign(a)(\vert a\vert - \kappa a)$ is a soft-thresholding. In (ii), $\kappa\in \RR$ is the \emph{unique} number such that the update satisfies the constraint. In (iv), $\kappa\geq 0$ is the \emph{smallest} number such that the update satisfies the constraint (see~\cite{condat2016fast} for efficient algorithms to compute $\kappa$ in practice). }\label{table:update}
\end{table}

Let us now recall the guarantees for these methods. We stress that, as discussed in Section~\ref{sec:general}, these guarantees do not necessarily lead to convergence rates.
\begin{proposition}\label{prop:classical-rates} Assume~(A1) and (A2) and that $\eta$ satisfies the conclusion of Lemma~\ref{lem:strong_convexity} for some $p\in {[1,2]}$, $\beta \geq 0$. Consider an initialization $f_0\in \dom H$ such that $\eta'(f_0)\in L^\infty(\tau)$ and, for some $K\geq \Vert f_0\Vert_{L^1(\tau)}$, a step-size $s \leq {(K+\beta)^{p-2}} (\Vert \Phi\Vert_\infty^2 \Lip(\nabla R))^{-1}$. 

(i) Let $(f_k)_{k\in \NN}$ be generated by Algorithm~\ref{alg:PGM} (PGM). If $\sup_k \Vert f_k\Vert_{L^1(\tau)} \leq K$, then Eq.~\eqref{eq:general-guarantee} holds with $\alpha_k= 1/(sk)$, i.e.
$$
F(f_k) - F(f) \leq \frac{1 }{s k} D_{\bar \eta}(f, f_0), \qquad \forall f\in L^1(\tau), \forall k\geq 1.
$$

(ii) Let $(f_k,g_k,h_k)_{k\in \NN}$ be generated by Algorithm~\ref{alg:APGM} (APGM). If $\sup_k \Vert h_k\Vert_{L^1(\tau)}\leq K$, then Eq.~\eqref{eq:general-guarantee} holds for $f_k$ with $\alpha_k=4/((k+1)^2s)$, i.e.
$$
F(f_k) - F(f) \leq \frac{4}{s (k+1)^2} D_{\bar \eta}(f, f_0), \qquad \forall  f\in L^1(\tau), \forall k\geq 1.
$$
Moreover, it holds $0<\gamma_k\leq 1$ and $\gamma_k\leq 2/(k+2)$ $\forall k \geq 0$.
\end{proposition}
\begin{proof}
By Proposition~\ref{prop:update}, the updates are well-defined. The proof of~\cite[Thm.~1]{tseng2010approximation} goes through, in particular thanks to Lemma~\ref{lem:smoothness} (smoothness) and since $D_\eta/(K+\beta)^{p-2}$ is $1$-strongly convex with respect to $\Vert \cdot \Vert_{L^1(\tau)}$ whenever this property is needed in the proof. A particularly simple exposition of the proof for APGM can be found in~\cite[Thm.~4.24]{d2021acceleration}. 
\end{proof}
\begin{remark}\label{rem:bound}
A difficulty in Proposition~\ref{prop:classical-rates} is that when $p<2$, one needs to assume \emph{a priori} bounds on the $L^1$-norm of certain iterates to obtain convergence guarantees, because the metric induced by the divergence $D_{\bar\eta}$ becomes weaker as the $L^1$-norm increases.
Since Algorithm~\ref{alg:PGM} (PGM) is a descent method, $\Vert f_k\Vert_{L^1}$ is bounded, uniformly in $k$, as soon as the objective is coercive for the $L^1$-norm. But for Algorithm~\ref{alg:APGM} (APGM), even if variants exist where $(F(f_k))_k$ is monotonous~\citep{d2021acceleration}, this property does not seem to be sufficient to control $\Vert h_k\Vert_{L^1}$, even for coercive objectives. Of course, uniform bounds are always trivially satisfied when $\bar H$ includes the constraint $\iota_{\Vert \mu \Vert\leq K}$ or $\iota_{\Pp(\Theta)}$.
\end{remark}

\subsection{Reparameterized gradient descent as a Bregman descent}\label{eq:reparameterized}
In this paragraph, we recall a link between Bregman gradient descent, a.k.a.~mirror descent (an instance of Algorithm~\ref{alg:PGM}) and $L^2$ gradient descent dynamics on certain \emph{reparameterized} objectives. 
The purpose is to show that the convergence rates proved in Section~\ref{sec:upper-bounds} with $\eta_{\mathrm{ent}}$ and $\eta_{\mathrm{hyp}}$ are also relevant to understand $L^2$ gradient descent in certain contexts.
While these remarks are well-known~\citep{amid2020winnowing,vaskevicius2019implicit,azulay2021implicit}, we find it instructive to state them clearly in our context. In order to reduce the discussion to its simplest setting, we consider the continuous time dynamics in the unregularized setting, and we assume that they are well-defined.

The optimality conditions of the update of Algorithm~\ref{alg:PGM} (see Proposition~\ref{prop:update}) can be written as $\eta'(f_{k+1})-\eta'(f_k) = - s F'[f_k]$. As the step-size $s$ vanishes, this leads to a continuous trajectory $(f_t)_{t\geq 0}$, which we refer to as the  $\eta$-\emph{mirror flow}, that solves
\begin{align}\label{eq:continuousMD}
\frac{d}{dt} \eta'(f_t)=-F'[f_t] &&\Leftrightarrow&& \frac{d}{dt} f_t = -[\eta''(f_t)]^{-1}F'[f_t].
\end{align}

\begin{proposition}[Reparameterized mirror flows as gradient flows] \label{prop:reparametrized}
(i) \emph{(Square parameterization).} Let $(f_t)_{t\geq 0}$ be the $L^2(\tau)$ gradient flow of $\hat F:f \mapsto F(f^2)$ initialized such that $\log(f_0)\in L^\infty(\tau)$. Then $h_t \coloneqq f_t^2$ is the $\eta_{\mathrm{ent}}$-mirror flow of $4F$.

(ii) \emph{(Difference of squares parameterization).} Let $(f_t,g_t)_{t\geq 0}$ be the $(L^2(\tau))^2$ gradient flow of $\hat F:(f,g) \mapsto F(f^2-g^2)$ initialized such that $\log(f_0g_0)\in L^\infty(\tau)$. Then $h_t \coloneqq f_t^2-g_t^2$ is the $\eta_{\mathrm{hyp}}$-mirror flow of $4F$ with parameter $\beta=2f_0g_0$ (here $\beta$ is function instead of a scalar). 
\end{proposition}
We can make the following remarks:
\begin{itemize}
\item Combining (i) and (ii), we find that if $(h^+_t, h^-_t)_{t\geq 0}$ follows a $\eta_{\mathrm{ent}}$-mirror flow for $(h^+,h^-)\mapsto F(h^+-h^-)$ then $h^+_t-h^-_t$ is a $\eta_{\mathrm{hyp}}$-mirror flow for $F$. This confirms the interpretation of $\eta_{\mathrm{hyp}}$ as a ``signed'' version of the entropy (see also~\cite[Thm.~23]{ghai2020exponentiated}).
\item These exact equivalences are lost in discrete time, with an error term that scales as the squared step-size. It is thus difficult to convert the most efficient guarantees for (Bregman) PGM into guarantees with the same convergence rate for gradient descent.
\end{itemize}

\begin{proof}
(i) The $L^2(\tau)$ gradient flow of $\hat F$ satisfies 
$
\frac{d}{dt} f_t = - \hat F'[f^2_t] =- 2 f F'[f^2_t].
$
Thus the function $h_t \coloneqq f_t^2$ evolves according to 
\begin{align*}
\frac{d}{dt} h_t = - 2f_t \frac{d}{dt} f_t = -4f_t^2 F'[f_t^2]= - 4 h_t F'[h_t]
\end{align*}
 which is precisely the $\eta_{\mathrm{ent}}$-mirror flow of $4F$ since $\eta_{\mathrm{ent}}''(s)=s^{-1}$ for $s>0$.

(ii) The $(L^2(\tau))^2$ gradient flow of $\hat F$ satisfies $\frac{d}{dt}  f_t = - 2 f_t  F'[f^2_t-g^2_t]$ and  $\frac{d}{dt}  g_t =  2 g_t F'[f^2_t-g^2_t]$. As a consequence, we have for $h_t \coloneqq f_t^2-g_t^2$ and $\tilde h_t \coloneqq f_t^2+g_t^2$ that 
$
\frac{d}{dt}  h_t = - 4 \tilde h_t F'[h_t].
$
To conclude, it remains to show that $\tilde h_t = [\eta_{\mathrm{hyp}}''(h_t)]^{-1} = (h_t^2+\beta^2)^{1/2}$ for some $\beta>0$. To prove this, observe that
\begin{align*}
\frac{d}{dt} (\tilde  h_t^2 - h_t^2) = \frac{d}{dt} (4f_t^2g_t^2) = 8 f'_tf_tg_t^2 + 8 g'_tg_tf_t^2 = 0.
\end{align*}
Hence $\tilde h_t^2 - h_t^2 = \tilde h_0^2 - h_0^2 = 4f_0^2g_0^2$, which proves that $\tilde h_t=(h_t^2+\beta^2)^{1/2}$ with $\beta=2f_0g_0$.
\end{proof}

\section{Upper bounds on the convergence rates}\label{sec:upper-bounds}

This section contains the main result of this paper which is Theorem~\ref{th:upper-bounds} and summarized in Table~\ref{table:rates}. As discussed in Section~\ref{sec:general} and thanks to Proposition~\ref{prop:classical-rates}, in order to derive convergence rates for Algorithms~\ref{alg:PGM} and~\ref{alg:APGM} it is sufficient to control the function
 \begin{align}\label{eq:psi}
  \psi(\alpha) \coloneqq \inf_{f\in L^1(\tau)} \Big\{ F(f)-\inf F + \alpha D_\eta(f,f_0)\Big\}.
 \end{align}
For the class of problems we consider, the behavior of $\psi$ highly depends on the context.  The simplest situation is when $F$ admits a minimizer $f^* \in L^q(\tau)$ with $q>1$ (since $\tau$ is finite, it holds $L^q(\tau)\subset L^1(\tau)$). Then for the distance-generating functions $\eta_p$ for $1<p\leq q$ or $\eta_{\mathrm{hyp}}$, it is easy to see that $D_{\bar \eta}(f^*,f_0)<+\infty$ (this further requires $f^*\geq 0$ for $\eta_{\mathrm{ent}}$). Thus by Proposition~\ref{prop:general-psi}, it holds $\psi(\alpha)\lesssim \alpha$ and the convergence rates $(\alpha_k)$ given in Proposition~\ref{prop:classical-rates} are preserved. 

In the more subtle case where the minimizer of $\bar F$ is only assumed to be in $\Mm(\Theta)$, the variety of behaviors is captured by the following result.

\begin{theorem}\label{th:upper-bounds} Under Assumptions~(A1) and (A2), let $f_0$ be such that $\eta'(f_0)\in L^\infty(\tau)$. Assume that there exists $\mu^* \in \Mm(\Theta)$ such that $\bar F(\mu^*)=\inf F$. Under setting $\text{(A2)}_{+}$ (i.e.~$\dom \eta=\RR_+$) assume moreover $\mu^*\in \Mm_+(\Theta)$.
Then it holds
\begin{align}\label{eq:rate-optim-form}
\psi(\alpha) \lesssim \inf_{\epsilon>0} \;\epsilon^q + \alpha \epsilon^d\eta(\epsilon^{-d}).
\end{align}
{where $q\in \{1,2,4\}$ is determined by Table~\ref{table:rates}-(b) (the largest $q$, the strongest the bound). Namely, the bound holds:
\begin{itemize}
\item with $q=1$ if $\Phi$ is Lipschitz continuous;
\item with $q=2$ if $\Phi$ is Lipschitz continuous and $\bar G'[\mu^*]=0$, or if $\nabla_\theta \Phi$ is Lipschitz continuous;
\item with $q=4$ if $\nabla_\theta \Phi$ is Lipschitz continuous and $\bar G'[\mu^*]=0$.
\end{itemize}}
\end{theorem}

Given the bound of Eq.~\eqref{eq:rate-optim-form}, it is straightforward to compute the rates given in Table~\ref{table:rates}. The exponent $q\in \{1,2,4\}$ can be interpreted a follows: it is such that $F(\mu_\epsilon)-F(\mu^*)\lesssim \epsilon^q$ where $\mu_\epsilon$ is the convolution of $\mu^*$ with a box kernel of radius $\epsilon>0$. An asymptotic analysis leads to lower bounds for this exponent under several assumptions (Table~\ref{table:rates}-(b)), but in practice, non asymptotic effects may play an important role (see experiments in Section~\ref{sec:numerics-nets}).

{
\begin{remark}[Additive vs. Multiplicative updates] A consequence of Theorem~\ref{th:upper-bounds} is that algorithms with ``additive updates'' -- obtained with $\eta_2$ as a distance-generating function (e.g.~ISTA, FISTA) -- suffer from the ``curse of dimensionality in the convergence rates, see Table~\ref{table:rates}-(a). In comparison, algorithms with ``multiplicative updates'' -- obtained with $\eta_{\mathrm{ent}}$ or $\eta_{\mathrm{hyp}}$ as a distance-generating function -- always converge at a faster rate which is independent of the dimension $d$.  Note that Theorem~\ref{th:upper-bounds} only proves upper bounds on the rates, but we will see that they are tight in 
Section~\ref{sec:lower-bounds}.
\end{remark}
}

\begin{table}
\centering
\begin{subtable}[h]{0.40\textwidth}\centering
\begin{tabular}{|c|c|c|}
\hline
%\diagbox[width=5.3cm]{Divergence}{Algorithm}
  & PGM & APGM\\
\hline
$\vphantom{\Big(}\eta_{\mathrm{ent}}, \eta_{\mathrm{hyp}}$ & $\log(k) k^{-1}$ & $\log(k) k^{-2}$\\
%$\vphantom{\Big(}\eta_2$ &  $k^{-\frac{q}{d+q}}$  & $k^{-\frac{2q}{d+q}}$ \\
$\vphantom{\Big(}\eta_p$, $p>1$ &  $k^{-\frac{q}{(p-1)d+q}}$ & $k^{-\frac{2q}{(p-1)d+q}}$  \\
\hline
\end{tabular}
\caption{Convergence rates}
\end{subtable}
\hfill
\begin{subtable}[h]{0.58\textwidth}\centering
\begin{tabular}{|c|r|r|}
\hline
& $\Phi$ Lipschitz & $\nabla_\theta \Phi$ Lipschitz \\
\hline
$\vphantom{\Big(}\bar G'[\mu^*]$ \text{arbitrary} & $q=1$\;\; {\color{gray}(I)} &  $q=2$ \;\; {\color{gray}(II)}\\
$\vphantom{\Big(} \bar G'[\mu^*]= 0$ &   $q=2$ {\color{gray}(I*)}  &  $q=4$ {\color{gray}(II*)}\\
\hline
\end{tabular}
\caption{Value of $q$}
\end{subtable}
\caption{(a) Upper bounds on the convergence rates of $F(f_k)-\inf F$ for Algorithm~\ref{alg:PGM} (PGM) and~\ref{alg:APGM} (APGM). (b) The value of $q$ that appears in the rate depends on the regularity of $\Phi$ and on whether $\bar G'$ vanishes at optimality or not. This defines $4$ settings referred to as (I), (I*), (II) and (II*). Upper bounds are derived in Thm.~\ref{th:upper-bounds}, lower bounds are proved in Section~\ref{sec:lower-bounds}.}\label{table:rates}
\end{table}

\begin{proof}
The upper-bound in Eq.~\eqref{eq:rate-optim-form} corresponds to an upper bound on $F(f_\epsilon)-\inf F + \alpha D_{\bar \eta}(f_\epsilon,f_0)$ for a specific family of candidates $f_\epsilon \in L^1(\tau)$. A special case of this argument for sparse $\mu^*$, $\eta_{\mathrm{ent}}$ and $q=1$ appeared in~\cite{chizat2021sparse} and extended to $\mu^*\in \Mm_+(\Theta)$ in~\cite{domingo2020mean}. In the following we write $F,G,H$ for $\bar F, \bar G, \bar H$ to lighten notations. 

\emph{Step 1. Smoothing with a box kernel.} For $\epsilon>0$ (smaller than the injectivity radius of the exponential map in $\Theta$) consider the transition kernel $(\gamma_{\theta,\epsilon})_{\theta \in \Theta}$ where $\gamma_{\theta,\epsilon}\in \Pp(\Theta)$ is proportional to the restriction of $\tau$ to the closed geodesic ball $B_{\epsilon}(\theta)$ of radius $\epsilon$ centered at $\theta$. We define $\gamma_{\epsilon}\in \Mm(\Theta^2)$ as $\gamma_\epsilon(\d \theta,\d \theta') \coloneqq \gamma_{\theta,\epsilon}(\d \theta ')\mu^*(\d \theta)$. By construction, the first marginal of $\gamma_\epsilon$ is $\mu^*$ and we call $\mu_\epsilon$ its second marginal, which is absolutely continuous with respect to $\tau$ with density $f_\epsilon=\frac{\d\mu_\epsilon}{\d\tau}$. Since $\frac{\d \gamma_{\theta,\epsilon}}{\d\tau}(\theta') = \tau(B_{\epsilon}(\theta))^{-1} \mathbf{1}_{B_{\epsilon}(\theta)}(\theta')$, it holds
\begin{align*}
\mu_\epsilon(\d \theta') =  \int_\Theta \gamma_{\theta,\epsilon}(\d \theta ')\mu^*(\d \theta), &&\text{and}&&  f_\epsilon(\theta') =\int_{\Theta} \frac{\mathbf{1}_{B_{\theta,\epsilon}}(\theta')}{\tau(B_{\epsilon}(\theta))}\d\mu^*(\theta).
\end{align*}
Note that $\tau (B_\epsilon(\theta)) \asymp \epsilon^d$, see e.g.~\cite[Thm.3.3]{gray1979riemannian}.

\emph{Step 2. Bounding $ F(\mu_\epsilon)- F(\mu^*)$.} For our admissible regularizers, it is easy to verify that $H(\mu_\epsilon)\leq  H(\mu^*)$. By convexity of $ G$, we have
\begin{align*}
 F(\mu_\epsilon) -  F(\mu^*) \leq  G(\mu_\epsilon) -   G(\mu^*)  \leq \int  G'[\mu_\epsilon]\d [\mu_\epsilon - \mu^*]= \int_\Theta \big(  G'[\mu_\epsilon](\theta) -  G'[\mu_\epsilon](\theta') \big)\d\gamma(\theta,\theta').
\end{align*}
It is clear that the magnitude and regularity of $ G'[\mu_\epsilon]$ plays a role in the magnitude of this quantity. To go further, let us consider the various cases in Table~\ref{table:rates}-(b) successively.

\textbf{(I).} If $\Phi$ is Lipschitz then for $\theta,\theta'\in \Theta$ it holds
\begin{align*}
\vert  G'[\mu_\epsilon](\theta) - G'[\mu_\epsilon](\theta')\vert \leq \Big\Vert \nabla R\Big(\int\Phi \d\mu_\epsilon \Big)\Big\Vert_\Ff \cdot \Lip(\Phi) \cdot \dist(\theta,\theta').
\end{align*}
Since $\nabla R$ is Lipschitz continuous, we deduce that there exists $K>0$ such that $\vert G'[\mu_\epsilon](\theta) - G'[\mu_\epsilon](\theta')\vert\leq K\dist(\theta,\theta')$. It follows
\begin{align*}
 F(\mu_\epsilon)- F(\mu^*) & \leq K \int_{\Theta\times \Theta} \dist(\theta,\theta')\d\gamma (\theta,\theta') \\
&= K\int_{\Theta}\frac{1}{\tau(B_{\epsilon}(\theta))}\int_{B_{\epsilon}(\theta)} \dist(\theta,\theta')\d\tau(\theta')\d\mu^\star(\theta)  \leq K\cdot \epsilon\cdot \Vert \mu^*\Vert.
\end{align*}

\textbf{(I*).} If $\Phi$ is Lipschitz and moreover $G'[\mu^*]=0$, it holds 
\begin{align*}
G'[\mu_\epsilon](\theta) = \Big \langle \nabla R\Big(\int \Phi \d\mu_\epsilon\Big)-\nabla R\Big(\int \Phi \d\mu^*\Big),\Phi(\theta)\Big\rangle_\Ff.
\end{align*}
Since $\nabla R$ is Lipschitz continuous, the first factor is bounded by $\Lip(\nabla R)\cdot \Vert \int \Phi \d[\mu_\epsilon-\mu^*]\Vert = \Lip(\nabla R)\cdot \sup_{\Vert \tilde \Phi\Vert\leq 1} \int \langle \tilde \Phi,\Phi(\theta)\rangle \d[\mu_\epsilon -\mu^*](\theta)$. Under assumption (I*), the functions $\{\theta\mapsto \langle \tilde \Phi,\Phi(\theta)\rangle\;;\; \Vert \tilde \Phi\Vert\leq 1\}$ are uniformly Lipschitz, so by the reasoning above, it follows that $\Vert \int \Phi \d[\mu_\epsilon-\mu^*]\Vert\lesssim \epsilon$, so $\Lip(G'[\mu_\epsilon]) \lesssim \epsilon$. It follows that the previous bound improves to $ F(\mu_\epsilon)- F(\mu^*)\lesssim \epsilon^2$.

\textbf{(II).} Here we have that $G'[\mu_\epsilon]\in \Cc^1(\Theta;\RR)$ and $\nabla_\theta G'[\mu_\epsilon]$ is Lipschitz. By the Mean Value Theorem on Riemannian manifolds (see e.g.~\cite[Thm.~4.6]{gray1982mean}), there exists a constant $K'\geq 0$ such that for all $\theta\in \Theta$
$$
\Big\vert \frac{1}{\tau(B_\epsilon(\theta))} \int_{B_{\epsilon}(\theta)} G'[\mu_\epsilon](\theta')\d\tau(\theta') - G'[\mu_\epsilon](\theta) \Big\vert \leq K' \epsilon^2.
$$
It follows
\begin{align*}
F(\mu_\epsilon)- F(\mu^*) & \leq \int \d\mu^*(\theta)\int \big( G'[\mu_\epsilon](\theta) - G'[\mu_\epsilon](\theta')\big) \d\gamma_{\theta,\epsilon}(\theta') \leq K' \Vert \mu^*\Vert\epsilon^2.
\end{align*}

\textbf{(II*).} If $\nabla \Phi$ is Lipschitz and moreover $G'[\mu^*]=0$ then an improvement as in (I*) applies. The functions $\{\theta\mapsto \langle \tilde \Phi,\Phi(\theta)\rangle\;;\; \Vert \tilde \Phi\Vert\leq 1\}$ are differentiable with a uniformly Lipschitz derivative so arguments as in the previous paragraph show that  $\Vert \int \Phi \d[\mu_\epsilon-\mu^*]\Vert\lesssim \epsilon^2$, so  $\Lip(\nabla_\theta G'[\mu_\epsilon])\lesssim \epsilon^2$. Thus, going through the argument for (II), with all the constants multiplied by $\epsilon^2$, we obtain that the bound improves to $F(\mu_\epsilon)- F(\mu^*)\lesssim \epsilon^4$.

\emph{Step 3. Bounding $D_{\bar \eta}(f_\epsilon,f_0)$.} It holds $D_{\bar \eta}(f_\epsilon,f_0) = \int \eta(f_\epsilon)\d\tau - \int \eta(f_0)\d\tau - \int \eta'(f_0)(f_\epsilon-f_0)\d\tau$. Since $\eta'(f_0)\in L^\infty$, all these terms are bounded by a constant independent of $\epsilon$ except $\int \eta(f_\epsilon)\d\tau$, so it remains to bound the latter.  If $\mu^*=0$ then this quantity is bounded by a constant independent of $\epsilon$ and we are done. Otherwise let us first assume that $\mu^*\in \Mm_+(\Theta)$. By Jensen's inequality, one has $\forall \theta'\in \Theta$
\begin{align*}
\eta(f_\epsilon(\theta')) 
&=  \eta \Big( \int_\Theta \frac{\mu^*(\Theta)\mathbf{1}_{B_{\epsilon}(\theta)}(\theta')}{ \tau(B_{\epsilon}(\theta))} \frac{\d\mu^*(\theta)}{\mu^*(\Theta)}\Big) 
\leq  \int_\Theta \eta \Big( \frac{\mu^*(\Theta)\mathbf{1}_{B_{\epsilon}(\theta)}(\theta')}{ \tau(B_{\epsilon}(\theta))} \Big)  \frac{\d\mu^*(\theta)}{\mu^*(\Theta)}.
\end{align*}
It follows, by Fubini's theorem, 
\begin{align*}
\int \eta(f_\epsilon(\theta'))\d\tau(\theta') 
&\leq \int_\Theta \frac{\d\mu^*(\theta)}{\mu^*(\Theta)} \int_\Theta \d\tau(\theta')  \eta \Big( \frac{\mu^*(\Theta)\mathbf{1}_{B_{\epsilon}(\theta)}(\theta')}{ \tau(B_{\epsilon}(\theta))} \Big)  \\
&\leq \left( \sup_\theta \tau(B_\epsilon(\theta))   \Big\vert \eta\Big( \frac{\mu^*(\Theta)}{\tau(B_{\epsilon}(\theta))} \Big) \Big\vert \right) + \vert \eta(0)\vert  
\end{align*}
Now we use the fact that $\sup_\theta \tau(B_{\epsilon}(\theta)) \asymp \epsilon^d$ and {our assumption that $\eta(cx)\asymp \eta(x)$ for any fixed $c>0$ as $x\to\infty$ to bound this quantity by $O(\epsilon^d \eta(\epsilon^{-d}))$}. 

For the general case where $\mu^*\in \Mm(\Theta)$ let $\mu^*_+,\mu^*_- \in \Mm_+(\Theta)$ be the Jordan decomposition of $\mu^* = \mu_+^*-\mu_-^*$. It holds $\vert f_\epsilon(\theta)\vert  \leq \max\{f_{+,\epsilon}(\theta),f_{-,\epsilon}(\theta)\}$ where $f_{+,\epsilon}$ and $f_{-,\epsilon}$ are obtained by applying the smoothing procedure of Step.~1 to $\mu^*_+$ and $\mu^*_-$ respectively. Using the fact that under Assumption~$\text{(A2)}_\pm$, $\eta$ is even and increasing on $\RR_+$, we obtain the bound
$$
\int \eta(f_\epsilon)\d\tau \leq \int \eta(f_{+,\epsilon})\d\tau+ \int \eta(f_{-,\epsilon})\d\tau.
$$
We finally bound each of these terms as when $\mu^*\in \Mm_+(\Theta)$ to conclude.
\end{proof}

\section{Lower bounds}\label{sec:lower-bounds}
We will consider two types of lower bounds: (i) lower bounds on $\psi(\alpha)$ in order to confirm that the analysis in Thm.~\ref{th:upper-bounds} is tight, and (ii) direct lower bounds on the convergence rates of Algorithms~\ref{alg:PGM} and~\ref{alg:APGM}. Of course, the latter imply the  former, but studying $\psi$ directly has its own interest and makes it simpler to cover all the cases.

\subsection{Tight lower bounds on $\psi$}
Let us show that the bounds on $\psi$ in Theorem~\ref{th:upper-bounds} cannot be improved without additional assumptions.

\begin{proposition}[Lower-bounds]\label{prop:lower-bound}
For each of the $4$ settings of Table~\ref{table:rates}-(b) (determining the value of $q\in\{1,2,4\}$), there exists an objective function $\bar F$ satisfying Assumption~(A1) and $\inf \bar F = \bar F(\mu^*)$ with $\mu^*\in \Mm_+(\Theta)$, such that for any distance-generating function $\eta$ satisfying Assumption~(A2) and $f_0$ such that $\eta'(f_0)\in L^\infty(\tau)$, it holds
\begin{align}\label{eq:rate-lower}
\psi(\alpha) \gtrsim\inf_{\epsilon>0} \;\epsilon^q + \alpha \epsilon^d\eta(\epsilon^{-d}).
\end{align}
\end{proposition}

\begin{proof}
Let us build explicit objective functions with $\Theta=\TT^d$ the $d$ dimensional torus, $\Ff=\RR$ and $\bar H=\iota_{\Pp(\Theta)}$. Let $\theta_0\in \TT^d$ and for $f\in L^1(\Theta)$ such that $f\tau \in \Pp(\Theta)$, let $\epsilon_f>0$ be such that $\int_{B_{\epsilon_f}(\theta_0)} f\d\tau = \frac12$. Such an $\epsilon_f$ exists because $\epsilon \mapsto \int_{B_\epsilon(\theta_0)} f\d\tau$ continuously interpolates between $0$ when $\epsilon=0$ and $1$ when $\epsilon$ is large. For any $\eta$ satisfying~Assumption (A2), using the fact that $\eta\geq 0$ and Jensen's inequality, it holds
\begin{align*}
\int \eta(f)\d\tau \geq  \int_{B_{\epsilon_f}(\theta_0)} \eta(f)\d\tau \geq \tau(B_{\epsilon_f}(\theta_0))   \eta\Big(\frac{1}{\tau(B_{\epsilon_f}(\theta_0))}\int_{B_{\epsilon_f}(\theta_0)} f\d\tau \Big) \gtrsim \epsilon_f^{d}\eta(\epsilon_f^{-d}).
\end{align*}
Thus for any $f_0$ such that $\eta'(f_0)\in L^\infty(\tau)$ it holds $D_{\bar \eta}(f,f_0)\gtrsim \epsilon_f^{d}\eta(\epsilon_f^{-d})$.

\textbf{(I).} Consider $\bar G(\mu) = \int_\Theta \dist(\theta_0,\theta)\d\mu(\theta)$. This satisfies Assumption~(A1)-(I) with $R$ the identity on $\RR$ and $\Phi(\theta)=\dist(\theta_0,\theta)$ which is clearly Lipschitz. Since $\bar H =\iota_{\Pp(\Theta)}$, $\bar F$ admits a unique minimizer $\mu^*=\delta_{\theta_0}$ and it holds $\bar F(\mu^*)=\inf \bar F=0$.
For any $f\in \dom H$, we have the following lower bound where $\epsilon_f$ is defined above:
\begin{align*}
F(f) -\inf F = \int_\Theta \dist(\theta,\theta_0)f(\theta)\d\tau(\theta) \geq \int_{\Theta \setminus B_{\epsilon_f}(\theta_0)} \dist(\theta,\theta_0)f(\theta)\d\tau(\theta) \geq \frac{\epsilon_f}2 .
\end{align*}
This proves the lower bound of Eq.~\eqref{eq:rate-lower} with $q=1$.

\textbf{(II).}  Consider  $\bar G(\mu) = \int_\Theta \tilde \Phi(\theta) \d\mu(\theta)$, where $\tilde \Phi$ is any function which is continuously twice differentiable, coincides with $\dist(\theta_0,\cdot)^2$ on $B_{1/2}(\theta_0)$ and is larger than $1/4$ outside of this ball. We cannot directly take $\dist(\theta_0,\cdot)^2$ because this function is not smooth everywhere on $\Theta$ due to the existence of a cut locus, but it is smooth on  $B_{1/2}(\theta_0)$. Assumptions (A1)-(II) are satisfied. Again $\mu^* =\delta_{\theta_0} $ is the unique minimizer and $F(\mu^*)=0$. Analogous computations show that $F(f)-\inf F \geq \min\{ \frac14, \epsilon_f^2\}$. This proves the lower bound of Eq.~\eqref{eq:rate-lower} with $q=2$.

\textbf{(I*).} Consider $\bar G(\mu) = \frac12 \big(\int \Phi \d\mu\big)^2$ where $\Phi(\theta)=\dist(\theta,\theta_0)$. Clearly, $\mu^* =\delta_{\theta_0} $ is the unique minimizer and $\bar G'[\mu^*]=0$ so Assumption (A1)-(I*) is satisfied. By direct computations, it holds $F(f)-\inf F \gtrsim \epsilon_f^2$. This proves the lower bound of Eq.~\eqref{eq:rate-lower} with $q=2$.

\textbf{(II*).}  Consider $\bar G(\mu) = \frac12 \big(\int \tilde \Phi \d\mu\big)^2$ where $\tilde \Phi$ is defined as in the analysis of (II). Again, $\mu^*$ is the unique minimizer and $\bar G'[\mu^*]=0$ so Assumption (A1)-(II*) is satisfied. By direct computations, it holds $F(f)-\inf F \gtrsim \epsilon_f^4$, which proves the lower bound with $q=4$.
\end{proof}

\begin{remark}[Exact decay of $\psi$ for a natural class of problems.] There exists in fact a broad class of problems satisfying Assumption (A1)-(II) for which the bound on $\psi$ with $q=2$ is exact. These are problems with a sparse solution $\mu^*$ that satisfy an additional non-degeneracy condition at optimality, that appear naturally in certain contexts~\citep{poon2018geometry}. For these problems, is it shown in~\cite[Prop.~3.2]{chizat2021sparse} that $\bar F(\mu) - \inf \bar F \gtrsim \big( \sup_{g} \int g \d[\mu - \mu^*]\big)^2$ where the supremum is over $1$-Lipschitz functions $g:\Theta\to \RR$ uniformly bounded by $1$. Reasoning as in the proof of Proposition~\ref{prop:lower-bound}, this implies $F(f) - \inf F \gtrsim \epsilon_f^2$ which leads to 
$$
\psi(\alpha) \asymp \inf_{\epsilon >0} \epsilon^2 + \alpha \epsilon^d \eta(\epsilon^{-d}).
$$ 
\end{remark}

\subsection{Direct lower bounds on the convergence rates}
In this section, we  directly lower-bound the convergence rates of Algorithm~\ref{alg:PGM} and Algorithm~\ref{alg:APGM}. We focus on the $L^2$ geometry ($\eta_{2}$) for which we prove all the lower bounds (there are $8$ cases to consider) and on the relative entropy geometry ($\eta_\mathrm{ent}$) for which we omit certain settings for the sake of conciseness. In all the cases considered, the lower bounds match the upper bounds (up to logarithmic terms for $\eta_\mathrm{ent}$). Let us start with PGM and $\eta_2$.

\begin{proposition}\label{prop:lower-eta2-PGM}
For each of the settings (I), (I*), (II) and (II*) under Assumption (A1), there exists a function $F$ such that the iterates of Algorithm~\ref{alg:PGM} (PGM) $(f_k)_{k\geq 0}$ with the distance-generating function $\eta_2$, and initialized with $f_0=1$, with any step-size $s>0$, satisfy
\begin{align*}
F(f_k) - \inf F\gtrsim k^{-\frac{q}{d+q}}
\end{align*}
where $q$ is the constant associated to the setting via Table~\ref{table:rates}-(b).
\end{proposition}
\begin{proof}
\textbf{(I).}  As in the proof of Proposition~\ref{prop:lower-bound}, we consider $\Theta=\TT^d$, $\theta_0 \in \TT$, $\bar H=\iota_{\Pp(\Theta)}$ and $\bar G(\mu) = \int \Phi \d\mu$ with $\Phi(\theta) = \dist(\theta,\theta_0)$. We set $s=1$ as the step-size plays no role in what follows. In this case, the update equation of Algorithm~\ref{alg:PGM} writes
$
f_{k+1} = \big( f_k - \Phi - \kappa_k \big)_+
$
where $ \kappa_k\in \RR$ is such that $f_{k+1}\in \dom H$. Thanks to the symmetries of the problem, a direct recursion shows that it holds $$f_k = (m(k) - k\Phi)_+$$ for some $m(k)\in \RR_+$. Let $r(k)\coloneqq m(k)/k$ which is such that $B_{r(k)}(\theta_0)$ is the support of $f_k$. We have
 \begin{align*}  
 1 = \int f_k\d\tau \asymp \int_0^{r(k)}u^{d-1} (m(k)-ku)\d u &\asymp \frac{m(k)}{d} r(k)^{d} - \frac{1}{d+1} k\cdot r(k)^{d+1}  \asymp m(k)^{d+1}k^{-d}
 \end{align*}
thus $m(k) \asymp k^{d/(d+1)}$ and $r(k)\asymp k^{-1/(d+1)}$. We can compute the objective
\begin{align*}
F(f_k) \asymp \int_0^{r(k)} u^{d-1} u (m(k)-ku)\d u &\asymp \frac{1}{d+1}m(k)r(k)^{d+1} - k\frac{1}{d+2}r(k)^{d+2} \asymp k^{-1/(d+1)}.
\end{align*}
Which proves the case (I) (here $q=1$) since $\inf F=0$. 

\textbf{(II).} Consider a smooth function $\tilde \Phi$ which equals $\dist(\theta,\theta_0)^2$ on the ball $B_{\nicefrac{1}{2}}(\theta_0)$ as in the proof of Proposition~\ref{prop:lower-bound}. Again the iterates have the form $f_k = (m(k)-k\tilde \Phi)_+$ and now let $r(k)^2=m(k)/k$. For $k$ large enough, so that $\tilde \Phi =\dist(\cdot,\theta_0)^2$ over $B_{r(k)}(\theta_0)$, it holds
 \begin{align*}
 1 = \int f_k\d\tau \asymp \int_0^{r(k)}u^{d-1} (m(k)-ku^2)\d u &\asymp \frac{m(k)}{d} r(k)^{d} - \frac{k}{d+2} \cdot r(k)^{d+2}  \asymp m(k)^{(d+2)/2}k^{-d/2}
 \end{align*}
thus $m(k) \asymp k^{d/(d+2)}$ and $r(k)\asymp k^{-1/(d+2)}$. It also holds
\begin{align*}
F(f_k) \asymp \int_0^{r(k)} u^{d-1} u^2 (m(k)-ku^2)\d u &\asymp \frac{m(k)}{d+2}r(k)^{d+2} - \frac{k}{d+4}r(k)^{d+4} \asymp k^{-2/(d+2)}
\end{align*}
Which proves the case (II) (here $q=2$).

\textbf{(I*).} We consider the function $G(\mu) = \frac12 (\int \Phi \d\mu)^2$. Now the reasoning is slightly more subtle because of the non-linearity. It holds $f_{k}=(m(k)-s(k)\Phi)_+$ where $s(k+1)=s(k) +\int \Phi f_k \d\tau$. Again, $m(k)$ is such that the iterate is feasible. Our upper bounds imply that $\int \Phi f_k \d\tau \lesssim k^{-1/(d+2)}$, thus 
$
s(k+1) - s(k) \lesssim k^{-1/(d+2)}
$
and it follows $s(k)\lesssim k^{\frac{d+1}{d+2}}$ and $F(f_k) \gtrsim (s(k)^{-1/(d+1)})^2\asymp k^{-2/(d+2)}$ as desired.

\textbf{(II*).} For the case (II*) we combine the ideas from (II) and (I*), let us just emphasize on the differences. We take the function $G(\mu) = \frac12 (\int \tilde \Phi )^2$. For $k$ large enough, it holds $f_k=(m(k)-s(k)\tilde \Phi)_+$ and, thanks to our upper-bound, $\int \tilde \Phi f_k \d\tau \lesssim k^{-2/(d+4)}$. The recursion becomes
$
s(k+1) - s(k) \lesssim  k^{-2/(d+4)}
$
and it follows $s(k)\lesssim k^{(d+2)/(d+4)}$ and $F(f_k) \gtrsim (s(k)^{-2/(d+2)})^2\asymp k^{-4/(d+4)}$ as desired.
\end{proof}

Let us now prove similar lower bounds for APGM, again for the specific choice of distance-generating function $\eta_2$.
\begin{proposition}
For each of the settings (I), (II), (I*) and (II*) under Assumption (A1), there exists a function $F$ such that the iterates of Algorithm~\ref{alg:APGM} (APGM) $(f_k)_{k\geq 0}$ with the distance-generating function $\eta_2$, and initialized with $f_0=1$, with any step-size $s>0$, satisfy
\begin{align*}
F(f_k) - \inf F\gtrsim k^{-\frac{2q}{d+q}}
\end{align*}
where $q$ is the constant associated to the setting via Table~\ref{table:rates}(b).
\end{proposition}

\begin{proof}
For \textbf{(I)}, we consider the same set up as in the proof of Prop.~\ref{prop:lower-eta2-PGM}-(I) (also fixing $s=1$ for conciseness). Since $G$ is linear, the update of $h_k$ reads $h_{k+1}=(m(k) - s(k)\Phi)_+$ with $s_k= \sum_{i=0}^{k}\frac{1}{\gamma_i}$ where $(\gamma_k)_k$ is defined in Algorithm~\ref{alg:APGM}. Since $\gamma_k \gtrsim 1/k$, we have $s(k)\asymp k^2$ which implies that $F(h_{k})\gtrsim k^{-2/(d+1)}$. Finally, it is clear that $f_k$ is in the convex hull of $\{h_0,\dots,h_k\}$ and since $G$ is linear, it holds $F(f_k)\geq \min_{i=1,\dots,k} F(h_k)\gtrsim k^{-2/(d+1)}$. The proof for the case \textbf{(II)} follows exactly the same scheme but with the function $\tilde \Phi$ considered in the proof of Proposition~\ref{prop:lower-eta2-PGM} and we omit the details.

In the case \textbf{(I*)}, we have $(m(k) -s(k)\Phi)_+$ with $s(k+1)-s(k)=\frac{1}{\gamma_k}\int \Phi g_k\d\tau$. Thanks to our upper-bounds, we have $\int \Phi f_k\d\tau\lesssim k^{-2/(d+2)}$, and since $\gamma_k\to 0$ and $\Vert h_k\Vert_{L^1}=1$, it follows  $\int \Phi g_k\d\tau\lesssim k^{-2/(d+2)}$. Thus $s(k+1)-s(k)\lesssim k^{d/(d+2)}$ which leads to $s(k)\lesssim k^{2(d+1)/(d+2)}$. It follows $F(h_k)\gtrsim s(k)^{-2/(d+1)} \gtrsim k^{-4/(d+2)}$. Since $h_k$ is optimal for $F$ over the convex hull of $\{h_0,\dots,h_k\}$, which contains $f_k$ (it has the most mass in small balls around $\theta_0$) we obtain the same lower bound on $F(f_k)$. The proof for the case \textbf{(II*)} follows the same scheme but with the function $\tilde \Phi$ and we omit the details.
\end{proof}

It is also instructive to look at lower bounds with the entropy $\eta_{\mathrm{ent}}$. We observe that here there exist cases where the guarantee given by Prop.~\ref{prop:classical-rates} is off by a $\log(k)$ factor (because this factor is present in the lower bound of Proposition~\ref{prop:lower-bound}).
\begin{proposition}
For the settings (I) and (II) under Assumption (A1), there exists a function $F$ such that the iterates of Algorithm~\ref{alg:PGM} (PGM) $(f_k)_{k\geq 0}$ with the distance-generating function $\eta_{\mathrm{ent}}$, and initialized with $f_0=1$, with any step-size $s>0$, satisfy
\begin{align*}
F(f_k) - \inf F\asymp k^{-1}.
\end{align*}
\end{proposition}

\begin{proof}
We consider the same setting as in the proof of Proposition~\ref{prop:lower-eta2-PGM}-(I) (and $s=1$ for simplicity). In this case, the update reads $f_{k+1} \propto \exp(f_k-\Phi)$ so by an immediate recursion $f_k\propto \exp(-k\Phi)$. This is essentially a (multi-dimensional) Laplace distribution and when $k$ is large, up to exponentially small terms in $k$, we can compute the integrals over $\RR^d$ instead of $\TT^d$. For the normalizing factor, we have
\begin{align*}
\int_{\RR^d} \exp(-k\Vert x\Vert_2^2)\d x \asymp \int_0^\infty u^{d-1}\exp(-k u)\d u = k^{2-d}\int_0^\infty s^{d-1}e^{-s}\d s = k^{2-d}\Gamma(d)
\end{align*}
where $\Gamma$ is the Gamma function. For the (unnormalized) value of $F(f_k)$, we have
\begin{align*}
\int_{\RR^d} \Vert x\Vert_2 \exp(-k\Vert x\Vert_2^2)\d x \asymp \int_0^\infty u^{d}\exp(-k u)\d u = \alpha^{1-d}\int_0^\infty s^{d}e^{-s}\d s = k^{1-d}\Gamma(d+1).
\end{align*}
By computing the ratio, it follows that $F(f_k)-\inf F \asymp k^{-1}$. 
In Setting (II), we take the function $\tilde \Phi$ as before which is equal to $\dist(\cdot,\theta_0)^2$ near $\theta_0$. Now $f_k\propto \exp(-k\tilde \Phi)$ which is essentially, when $k$ is large, a Gaussian distribution of variance $1/k \asymp F(f_k)- \inf F$.
\end{proof}

\begin{remark} Although the convergence rates obtained with $\eta_{\mathrm{ent}}$ and $\eta_{\mathrm{hyp}}$ are independent of the dimension $d$ (see Table~\ref{table:rates}), this favorable behavior crucially relies on the assumption that $\bar F$ admits a minimizer $\mu^*\in \Mm(\Theta)$. When this is not the case,~\cite{wojtowytsch2020can} show that there is an example where the continuous time dynamics induced by $\eta_{\mathrm{ent}}$ also suffer from the curse of dimensionality (our setting is slightly different but their argument would apply here). In addition, the discrete time dynamics are not stable in this case because the norm of the iterates grows unbounded, see Remark~\ref{rem:bound}.
\end{remark}

\section{Numerical experiments}\label{sec:numerics}
In this section we compare our theoretical rates with the practical behavior of PGM (Algorithms~\ref{alg:PGM}) and APGM (Algorithm~\ref{alg:APGM}) on simple toy problems. The purpose is to show that, although our analysis is asymptotic (in $k$ and in the spatial discretization), it describes well the convergence of those algorithms in certain practical scenarios. The code to reproduce these experiments can be found online\footnote{\url{https://github.com/lchizat/2021-measures-PGM}}.

\begin{figure}
\centering
\includegraphics[scale=0.5]{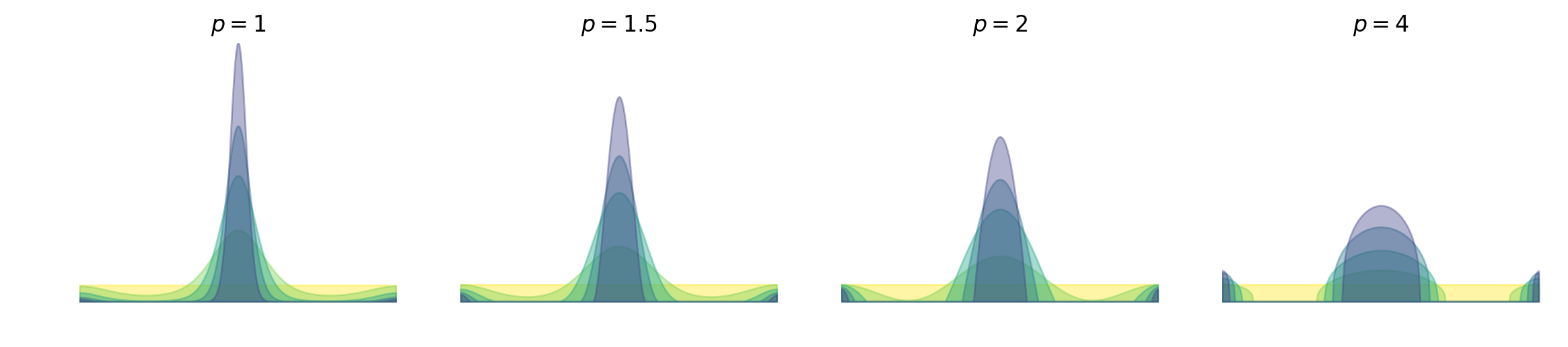}
\caption{Behavior of Algorithm~\ref{alg:PGM} (PGM) for a nonnegative sparse deconvolution problem with solution $\mu^*=\delta_{0}$, with $d=1$ and for various Bregman divergences $\eta_p$ ($p=1$ stands for $\eta_{\mathrm{ent}}$).  We plot the function $f_{k}$ for $k\in \{0,6,6^2,6^3,6^4\}$,  we use the same step-size and the same axes in all cases. The associated convergence plots are in Figure~\ref{fig:1d-pos}.}\label{fig:illustration-cvgce}
\end{figure}

\begin{figure}
\centering
\begin{subfigure}{0.45\linewidth}
\centering
\includegraphics[scale=0.6]{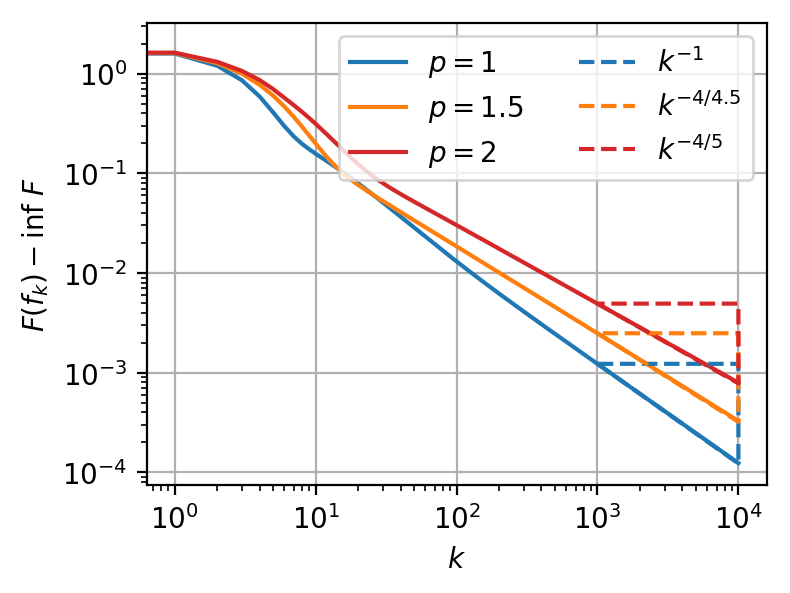}
\caption{PGM ($d=1$, $q=4$)}
\end{subfigure}
\begin{subfigure}{0.45\linewidth}
\centering
\includegraphics[scale=0.6]{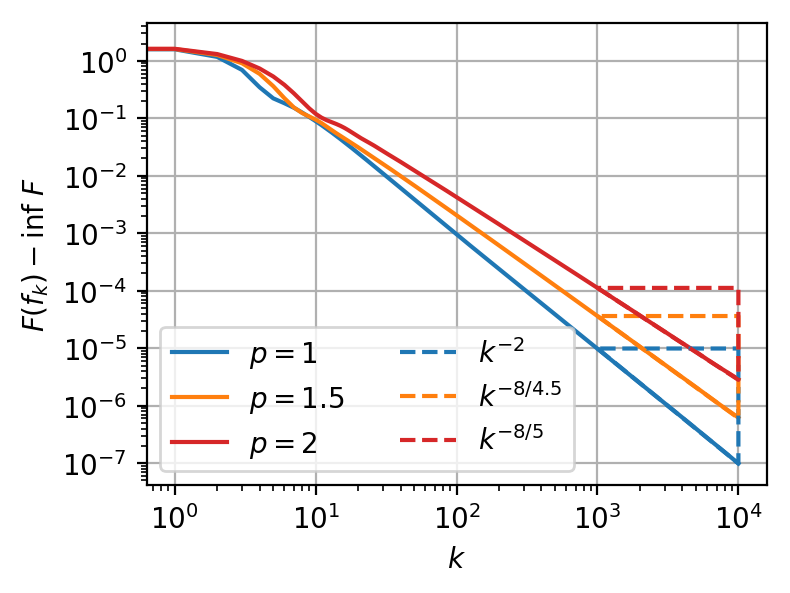}
\caption{APGM ($d=1$, $q=4$)}
\end{subfigure}
\caption{Convergence of PGM and APGM vs.~theoretical rates (up to log factors) in a sparse deconvolution problem with $\bar H=\iota_{\Mm_+(\Theta)}$ and $d =1$. Here $p$ refers to the parameter of $\eta_p$ and $p=1$ refers to $\eta_{\mathrm{ent}}$.  The objective has structure (II*) so $q=4$ in the rates of Table~\ref{table:rates}.}\label{fig:1d-pos}
\end{figure}

\begin{figure}
\centering
\begin{subfigure}{0.45\linewidth}
\centering
\includegraphics[scale=0.6]{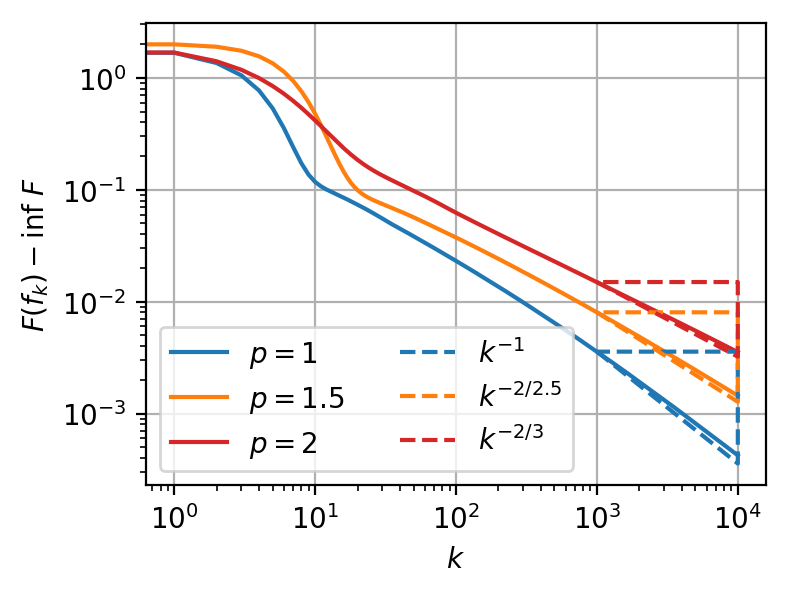}
\caption{PGM ($d=1$, $q=2$)}
\end{subfigure}
\begin{subfigure}{0.45\linewidth}
\centering
\includegraphics[scale=0.6]{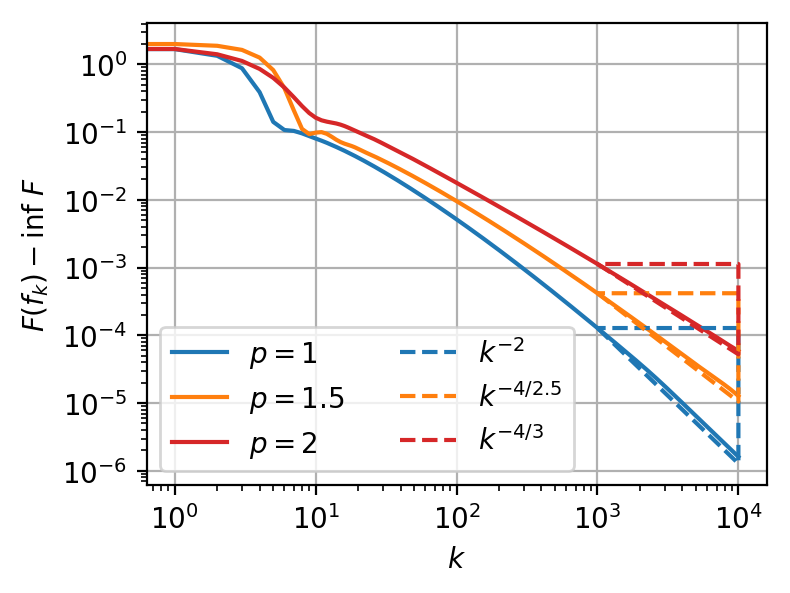}
\caption{APGM ($d=1$, $q=2$)}
\end{subfigure}
\caption{Convergence of PGM and APGM vs.~theoretical rates (up to log factors) in a sparse deconvolution problem with $\bar H=\lambda \Vert \mu\Vert$ and $d =1$. Here $p$ refers to the parameter of $\eta_p$ and $p=1$ refers to $\eta_{\mathrm{hyp}}$.  The objective has structure (II) so $q=2$ in the rates of Table~\ref{table:rates}.}\label{fig:1d-sign}
\end{figure}

\begin{figure}
\centering
\begin{subfigure}{0.45\linewidth}
\centering
\includegraphics[scale=0.6]{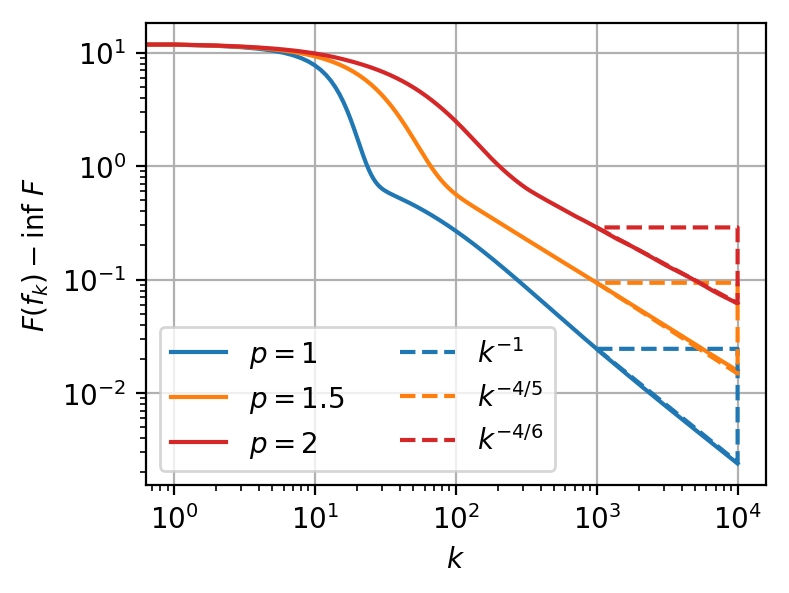}
\caption{PGM ($d=2$, $q=4$)}
\end{subfigure}
\begin{subfigure}{0.45\linewidth}
\centering
\includegraphics[scale=0.6]{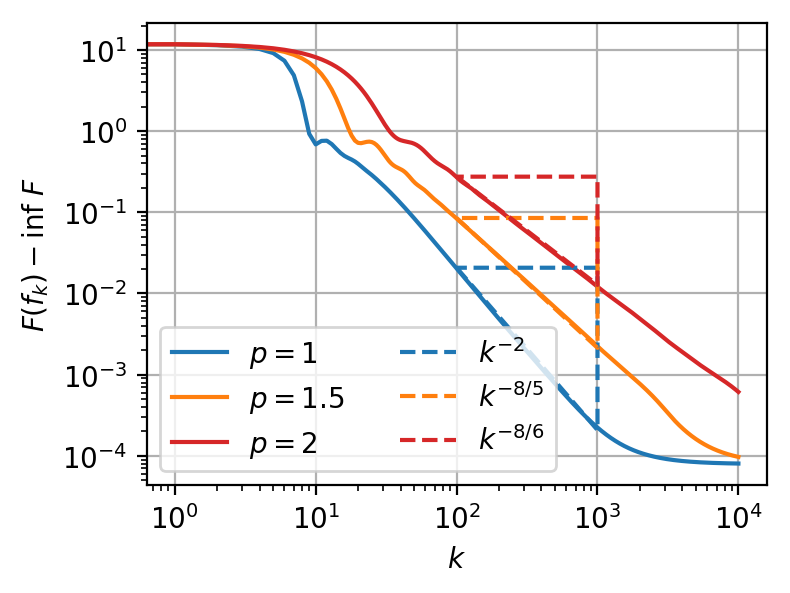}
\caption{APGM ($d=2$, $q=4$)}
\end{subfigure}
\caption{Convergence of PGM and APGM vs.~theoretical rates (up to log factors) in a sparse deconvolution problem with $\bar H=\iota_{\Mm_+(\Theta)}$ and $d =2$. Here $p$ refers to the parameter of $\eta_p$ and $p=1$ refers to $\eta_{\mathrm{ent}}$.  The objective has structure (II*) so $q=4$ in the rates of Table~\ref{table:rates}.}\label{fig:2d-pos}
\end{figure}

\begin{figure}
\centering
\begin{subfigure}{0.45\linewidth}
\centering
\includegraphics[scale=0.6]{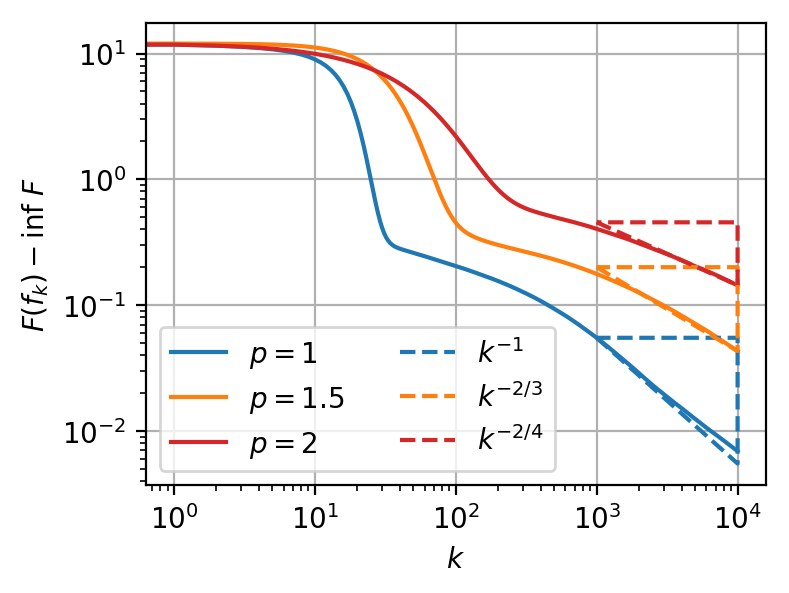}
\caption{PGM ($d=2$, $q=2$)}
\end{subfigure}
\begin{subfigure}{0.45\linewidth}
\centering
\includegraphics[scale=0.6]{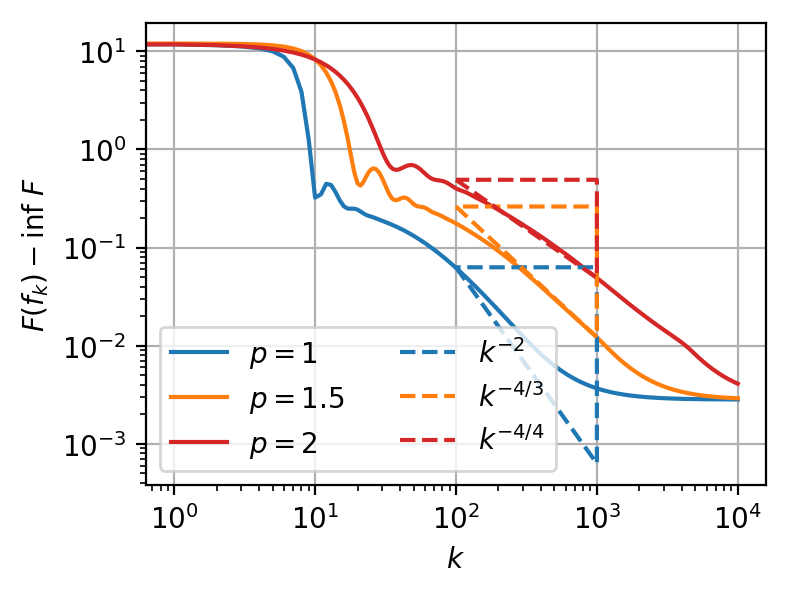}
\caption{APGM ($d=2$, $q=2$)}
\end{subfigure}
\caption{Convergence of PGM and APGM vs.~theoretical rates (up to log factors) in a sparse deconvolution problem with $\bar H=\lambda \Vert \mu\Vert$ and $d =2$. Here $p$ refers to the parameter of $\eta_p$ and $p=1$ refers to $\eta_{\mathrm{hyp}}$.  The objective has structure (II) so $q=2$ in the rates of Table~\ref{table:rates}.}\label{fig:2d-sign}
\end{figure}

\subsection{Sparse deconvolution}
We consider the sparse deconvolution problem introduced in Section~\ref{sec:objective} where $\phi$ is a Dirichlet kernel $\phi(\theta) = \sum_{k \in \{-2,1,0,1,2\}^d} \exp(2\sqrt{-1}\pi k^\top \theta)$ and $y^*(\theta)= \phi(\theta) = \int \phi(\theta-\theta')\d\mu^*(\theta')$ with $\mu^*=\delta_{0}$. The domain $\TT^1$ is discretized into a regular grid of $m=300$ points and $\TT^2$ into a regular grid of $m=60\times 60$ points. Figure~\ref{fig:illustration-cvgce}  illustrates the behavior of the various Bregman divergences for this problem, where it is seen that the iterates $f_k$ (weakly) converge faster to the Dirac solution as $p$ is smaller (in the following discussion, we use $p=1$ to refer to the entropy or hypentropy distance-generating function). 

Figures~\ref{fig:1d-pos}, \ref{fig:1d-sign}, \ref{fig:2d-pos} and \ref{fig:2d-sign} report the convergence rates in a variety of settings, which we compare to our theoretical predictions (without the logarithmic factors, since they do not change the asymptotic slope on a log-log plot). In both cases, $\inf F$ admits a closed form so we can exactly plot $\bar F(f_k\tau_m)-\min \bar F$ and additionally observe the effect of the discretization (here $\tau_m$ is the discretized reference measure). Observe that in the 2D experiments, APGM with $p=1$ quickly reaches the discretization error, and on Figure~\ref{fig:2d-sign}-(b), it does not have enough ``time'' to attain the theoretical asymptotic rate before the effect of the discretization comes in. While our analysis is asymptotic, it thus corresponds in practice to a non-asymptotic and transient behavior.

\subsection{Two-layer neural networks}\label{sec:numerics-nets}
We consider a two-layer ReLU neural network with the objective function introduced in Section~\ref{sec:objective} where we consider $n=10$ input samples $x_i$ on a regular grid on $[-1,1]$ and observed variables  $y_i = \vert x_i\vert -\frac12 + Z_i$ where $Z_i$ are independent and uniform on $[-1,1]$ (see the samples on Figure~\ref{fig:NN-illustration}-(b)). The domain is $\SS^1$ discretized into a regular grid of $m=2000$ points. 
This setting gives an example where $\Phi$ does not have a Lipschitz gradient and is only Lipschitz (observe the irregularity of $\bar G'[\mu^*]$ on Figure~\ref{fig:NN-illustration}-(a)). Since we use the regularization $\bar H=\lambda \Vert \mu\Vert$, we are in the setting (I) from Table~\ref{table:rates}-(b), and the parameter for the rate is $q=1$. 

Figure~\ref{fig:NN-rates} shows the rates of convergence for PGM and APGM. Although the general picture is consistent with the theory, we observe that our guarantees are a bit over-conservative. For PGM, we roughly measure (between iteration $k=10^3$ and $k=10^5$) the rates exponents $(-1.00, -0.72,-0.58)$ for respectively $p=(1,1.5,2)$ which corresponds to a parameter $q\approx 1.5$ rather than $q=1$. For APGM, we roughly measure the rates exponents $(-1.97, -1.71,-1.41)$ instead of the predicted $(-2,-1.33,-1)$. Figure~\ref{fig:NN-illustration}-(a) helps understanding this discrepancy: as can be seen from the proof of Theorem~\ref{th:upper-bounds} what truly determines the asymptotic rate is {how much the objective function increases when $\mu^*$ is mollified, and we quantified this using the regularity and magnitude of $\bar G'[\mu]$ near $\mu^*$. Here it appears that $\bar G'[\mu^*]$ is smooth near 2 out of the 3 points in the support of $\mu^*$, while it is non-smooth at the third point (the one in the middle). The fact that we have a mix of both levels of regularity (i.e.~smooth vs.~merely Lipschitz) may explain why the convergence is a bit faster than with the parameter $q=1$, which corresponds to only taking into account the Lipschitz regularity.}

\begin{figure}
\centering
\begin{subfigure}{0.45\linewidth}
\centering
\includegraphics[scale=0.6]{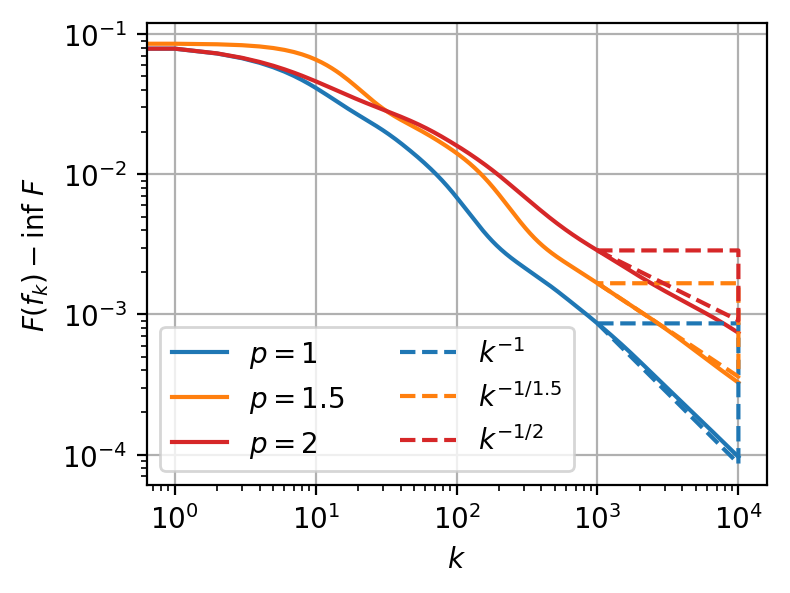}
\caption{PGM ($d=2$, $q=2$)}
\end{subfigure}
\begin{subfigure}{0.45\linewidth}
\centering
\includegraphics[scale=0.6]{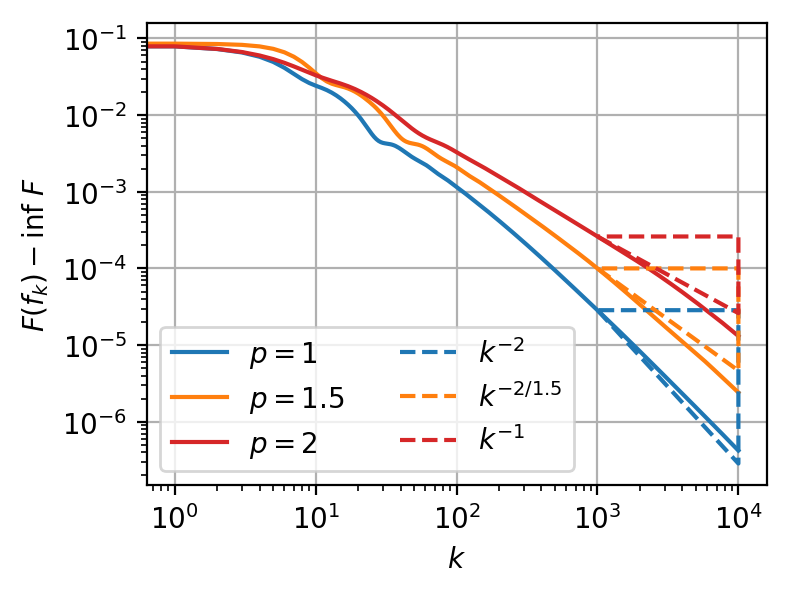}
\caption{APGM ($d=2$, $q=2$)}
\end{subfigure}
\caption{Convergence of PGM and APGM vs.~theoretical rates (up to log factors) for a $2$-layer ReLU neural network with $\bar H=\lambda \Vert \mu\Vert$ and $d =1$. Here $p$ refers to the parameter of $\eta_p$ and $p=1$ refers to $\eta_{\mathrm{hyp}}$.  The objective has structure (I) so $q=1$ in the rates of Table~\ref{table:rates}.}\label{fig:NN-rates}
\end{figure}

\begin{figure}
\centering
\begin{subfigure}{0.64\linewidth}
\centering
\includegraphics[scale=0.46]{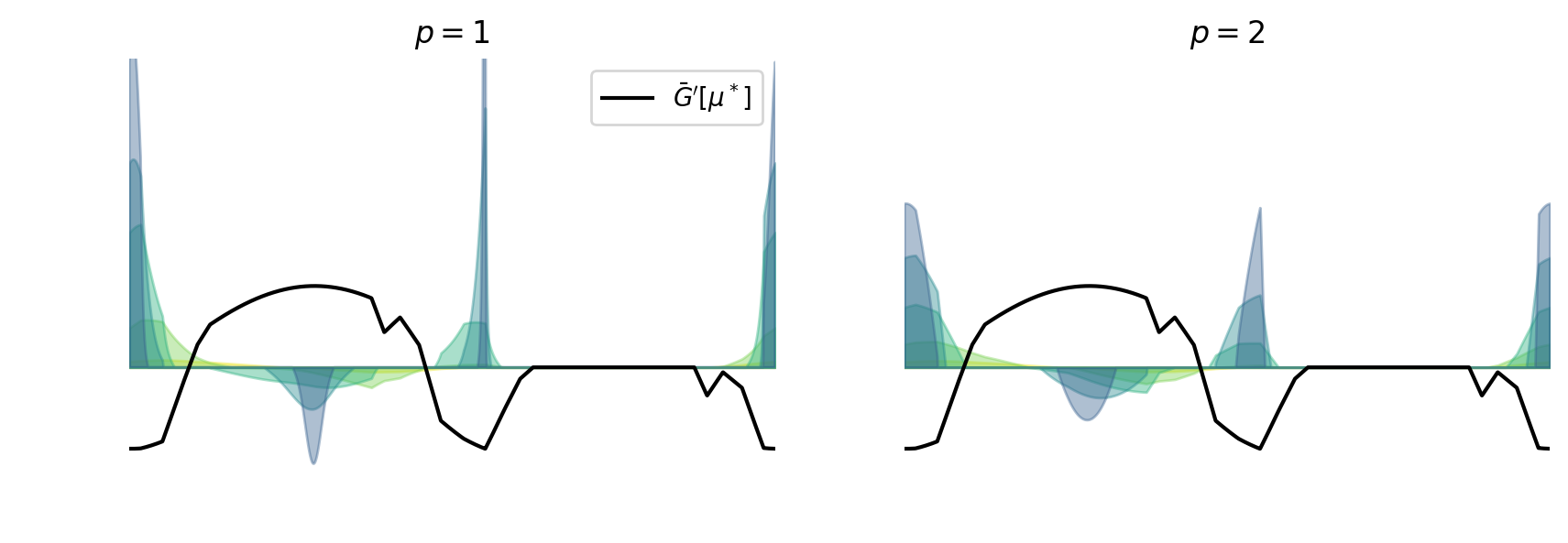}
\caption{Densities $f_k$ for $k\in \{6,6^2,6^3,6^4,6^5\}$.}
\end{subfigure}
\begin{subfigure}{0.34\linewidth}
\centering
\includegraphics[scale=0.46]{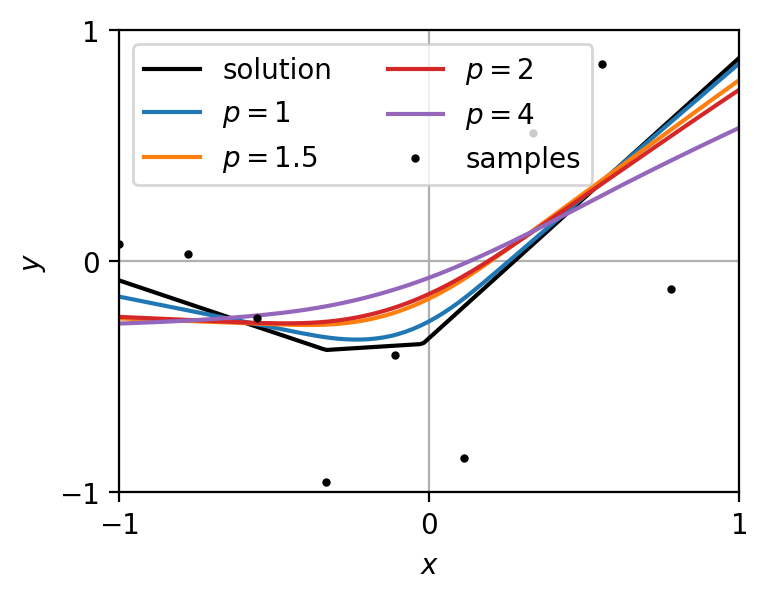}
\caption{Regressor at $k=200$.}
\end{subfigure}
\caption{Dynamics of PGM on a two-layer neural network, for various values of $p$ ($p=1$ corresponds to $\eta_{\mathrm{hyp}}$). Observe in (b) how the dynamics with $\eta_{\mathrm{hyp}}$ fits the kinks of the optimal regressor much faster than with $p>1$. }\label{fig:NN-illustration}
\end{figure}

\section{Conclusion}
We have studied the convergence rates of PGM and APGM for convex optimization in the space of measures. Our analysis exhibits the influence of the regularity of the objective function on the convergence rates. It also confirms that the geometry induced by $\eta_{\mathrm{ent}}$ and $\eta_{\mathrm{hyp}}$ is better suited than the $L^2$ geometry to solve such problems. An important question for future research is to better understand the unregularized case, where the phenomenon of algorithmic regularization is at play.

\paragraph{Acknowledgments.} I am thankful to Adrien Taylor for fruitful discussions during the preparation of this paper. In particular, I learnt about Algorithm~\ref{alg:APGM} (APGM) from him.
\bibliography{LC.bib}

%\appendix

\end{document}